\documentclass[10pt]{article}
\usepackage{latexsym}
\usepackage{amsmath}
\usepackage{amssymb}
\usepackage{amsthm}
\usepackage{cas-math}
\usepackage[final]{epsfig}
\usepackage[matrix,arrow,cmtip,curve]{xy}
\usepackage{color}

\newcommand{\qm}{\bullet}

\usepackage[textwidth=14cm,textheight=21cm]{geometry}
\usepackage{prettyref}[1998/07/09]
\usepackage[final]{hyperref}

\makeatletter
\renewcommand\theenumi{\@roman\c@enumi}
\renewcommand\theenumiii{\@Roman\c@enumiii}
\makeatother

\newrefformat{thm}{\hyperref[#1]{Theorem~\ref*{#1}}}
\newrefformat{prop}{\hyperref[#1]{Proposition~\ref*{#1}}}
\newrefformat{cor}{\hyperref[#1]{Corollary~\ref*{#1}}}
\newrefformat{eq}{\hyperref[#1]{\textup{(\ref*{#1})}}}
\newrefformat{it}{\hyperref[#1]{\textup{(\ref*{#1})}}}
\newrefformat{eqn}{\hyperref[#1]{\textup{(\ref*{#1})}}}
\newrefformat{sec}{\hyperref[#1]{Section~\ref*{#1}}}
\newrefformat{defn}{\hyperref[#1]{Definition~\ref*{#1}}}
\newrefformat{def}{\hyperref[#1]{Definition~\ref*{#1}}}
\newrefformat{lem}{\hyperref[#1]{Lemma~\ref*{#1}}}
\newrefformat{lemma}{\hyperref[#1]{Lemma~\ref*{#1}}}
\newrefformat{rem}{\hyperref[#1]{Remark~\ref*{#1}}}
\newrefformat{ex}{\hyperref[#1]{Example~\ref*{#1}}}
\newrefformat{fig}{\hyperref[#1]{Figure~\ref*{#1}}}
\newrefformat{q}{\hyperref[#1]{Question~\ref*{#1}}}
\newrefformat{conj}{\hyperref[#1]{Conjecture~\ref*{#1}}}

\definecolor{draftmargin}{rgb}{0.7,0,0}

\newcommand{\looped}[1]{{#1}^\circ}
\DeclareMathOperator{\Chain}{Chain}
\def\real{\abs}
\newcommand{\ug}{{\mathbf 1}}  % unit graph

\newcommand{\Poset}{\mbox{\upshape Poset}}
\DeclareMathOperator{\coind}{coind}
\DeclareMathOperator{\ind}{ind}
\DeclareMathOperator{\conn}{conn}
\DeclareMathOperator{\Ht}{ht}
\newcommand{\htt}{\Ht_{\Z_2}}

\newcommand{\Cpm}{{\mathcal C_{2m}^1}}
\newcommand{\Cm}{{\mathcal C_{2m}}}

\newcommand{\Cptx}[1]{{\mathcal C_{#1}^1}}

\begin{document}

\theoremstyle{plain}
    \newtheorem{thm}{Theorem}[section]
    \newtheorem*{thm*}{Theorem}
    \newtheorem{prop}[thm]{Proposition}
    \newtheorem{lemma}[thm]{Lemma}
    \newtheorem{lem}[thm]{Lemma}
    \newtheorem{conj}[thm]{Conjecture}
    \newtheorem{cor}[thm]{Corollary}

\theoremstyle{definition}
    \newtheorem{defn}[thm]{Definition}
    \newtheorem{nota}[thm]{Notation}

\theoremstyle{remark}
    \newtheorem{rem}[thm]{Remark}
    \newtheorem{example}[thm]{Example}
    \newtheorem{question}[thm]{Question}

% Article information
\title{Topology of Hom complexes and test graphs\\
for bounding chromatic number}
\date{July 29, 2009}

%OLD TITLE:\title{Constructing test graphs for bounding chromatic number in Hom complexes}

% Author information

\author{Anton Dochtermann and Carsten Schultz \\
\small Technische Universit\"at Berlin, MA 6-2 \\[-0.8ex]
\small Stra\ss e des 17. Juni 136, 10623 Berlin, Germany\\[-0.8ex]
\small {\tt anton.dochtermann@gmail.com, carsten@codimi.de} }

\maketitle

\begin{abstract}
The $\Hom$ complex of homomorphisms between two graphs was originally introduced to provide topological lower bounds on chromatic number.  In this paper we introduce new methods for understanding the topology of $\Hom$ complexes, mostly in the context of $\Gamma$-actions on graphs and posets (for some group $\Gamma$).  We view the $\Hom(T,\qm)$ and $\Hom(\qm,G)$ complexes as functors from graphs to posets, and introduce a functor $(\qm)^1$ from posets to graphs obtained by taking atoms as vertices.  Our main structural results establish useful interpretations of the equivariant homotopy type of $\Hom$ complexes in terms of spaces of equivariant poset maps and $\Gamma$-twisted products of spaces.  When $P \deq F(X)$ is the face poset of a simplicial complex $X$, this provides a useful way to control the topology of $\Hom$ complexes.  These constructions generalize those of the second author from \cite{Schspace} as well as the calculation of the homotopy groups of $\Hom$ complexes from \cite{DocGro}.

Our foremost application of these results is the construction of new families of \emph{test graphs} with arbitrarily large chromatic number - graphs $T$ with the property that the connectivity of $\Hom(T,G)$ provides the best possible lower bound on the chromatic number of $G$.  In particular we focus on two infinite families, which we view as higher dimensional analogues of odd cycles.  The family of \emph{spherical graphs} have connections to the notion of \emph{homomorphism duality}, whereas the family of \emph{twisted toroidal graphs} lead us to establish a weakened version of a conjecture (due to Lov\'{a}sz) relating topological lower bounds on chromatic number to maximum degree.  Other structural results allow us to show that any finite simplicial complex $X$ with a free action by the symmetric group $S_n$ can be approximated up to $S_n$-homotopy equivalence as $\Hom(K_n,G)$ for some graph $G$; this is a generalization of the results of Csorba from \cite{Cs05} for the case of $n=2$.  We conclude the paper with some discussion regarding the underlying categorical notions involved in our study.
\end{abstract}

\section{Introduction}

\subsection{Some background}

In his 1978 proof of the Kneser conjecture, Lov\'{a}sz \cite{Lov78} showed that
the chromatic number of a graph is bounded below by the connectivity
of (a complex later shown to be homotopy equivalent to) $\Hom(K_2, G)$, a space of homomorphisms from the edge $K_2$ into $G$.  Some 25 years later, Babson and Kozlov \cite{BKpro} were able to show that the
connectivity of $\Hom(C_{2r+1}, G)$ provided the next natural bound
on the chromatic number of $G$ (here $C_{2r+1}$ is an odd cycle), answering in the affirmative a conjecture of
Lov\'{a}sz.  When $G$ is a loopless graph (as is the case when we consider the
chromatic number), both $\Hom(K_2, G)$ and $\Hom(C_{2r+1}, G)$ are naturally
free ${\mathbb Z}_2$-spaces, and one can consider a list of
numerical invariants that measure the complexity of this action. The
now standard proof of the Lov\'{a}sz criterion gives the following
result.

\begin{thm}[Lov\'{a}sz] \label{thm:Lovtheorem}
 For every graph~$G$,
\begin{center}
$\ind_{{\mathbb Z}_2} \Hom(K_2, G) \leq \chi(G) - 2$.
\end{center}
\end{thm}

Here, for a free ${\mathbb Z}_2$-space $X$,  $\ind_{{\mathbb Z}_2}
X$ is the smallest dimension of a sphere with the antipodal action
that $X$ maps into equivariantly.  Since $\conn (X) + 1 \leq
\ind_{{\mathbb Z}_2}X$, this implies the original result of
Lov\'{a}sz from \cite{Lov78}.

As a way to take advantage of the ${\mathbb Z}_2$-topology, Babson
and Kozlov introduced the use of characteristic classes to the
study of $\Hom$ complexes.  In \cite{BKpro} they proposed and partially (according
to the parity of $\chi(G)$) proved the following result
incorporating the ${\mathbb Z}_2$-action on the $\Hom$ complex.

\begin{thm} \label{thm:strong}
For every graph~$G$,
\begin{center}
$\Ht_{{\mathbb Z}_2}\Hom(C_{2r+1}, G) \leq \chi(G)-3$.
\end{center}
\end{thm}

Here, for a free ${\mathbb Z}_2$-space $X$, $\Ht_{{\mathbb Z}_2}(X)$
is the highest nonvanishing power of the first Stiefel-Whitney class
of the bundle ${\mathbb R} \times_{{\mathbb Z}_2} X \rightarrow X/{\mathbb Z}_2$ (called the \emph{height} or sometimes the \emph{cohomological index} of $X$). Again, since $\conn(X) + 1 \leq \Ht_{{\mathbb
Z}_2}$X, this implies the connectivity bound that they did succeed
in proving completely.

The first complete proof of \prettyref{thm:strong} was given by the second author in \cite{Sch}.  More recently, in \cite{Schspace}, the
same author was able to prove the following statement which not only
implies \prettyref{thm:strong} but also provides insight into the
structure of the $\Hom$ complexes and suggests extensions that are
the theme to this paper.

\begin{thm}[Schultz] \label{thm:limitSch}
For every graph~$G$,
\begin{center}
$\colim_{\text{r}} \Hom(C_{2r+1}, G) \simeq_{{\mathbb Z}_2}
\Map_{{\mathbb Z}_2}\bigl({\mathbb S}^1_b, \Hom(K_2, G)\bigr)$.
\end{center}
\end{thm}

The direct system that defines the colimit is obtained by applying
the $\Hom(\qm,G)$ functor to the system $\cdots \rightarrow C_{2r+3}
\rightarrow C_{2r+1} \rightarrow \cdots$.  Here ${\mathbb S}^k_b$
denotes the $k$-sphere with the ${\mathbb Z}_2$-antipodal action on the
left and the ${\mathbb Z}_2$-reflection action on the right (which
in total can be considered an action of ${\mathbb Z}_2 \times
{\mathbb Z}_2$).  One then observes that $\Hom(K_2,
C_{2r+1}) \simeq_{{\mathbb Z}_2 \times {\mathbb Z}_2} {\mathbb
S}^1_b$, where the action on the first space is induced by the
nonidentity automorphism of $K_2$ and by the reflection of $C_{2r+1}$ that flips an edge.  The relevance of ${\mathbb S}^k_b$
in this context is exhibited in the following result, a proof of which can be found in \cite{Schspace}.

\begin{prop}\label{schultzprop}\label{prop:Sk}
If $X$ is a ${\mathbb Z}_2$-space and $k\ge0$, then $\Ht_{{\mathbb
Z}_2}(\mathbb{S}^k_b \times_{{\mathbb Z}_2} X) \geq \Ht_{{\mathbb
Z}_2} X + k$.
\end{prop}

\noindent One can then combine these observations to obtain the
following corollary which, when combined with \prettyref{thm:Lovtheorem}, implies \prettyref{thm:strong}.
\begin{cor}
If $G$ is a graph with at least one edge, then
\begin{center}
$\Ht_{{\mathbb Z}_2} \Hom(C_{2r+1}, G) + 1 \leq \Ht_{{\mathbb Z}_2}
\Hom(K_2, G)$.
\end{center}
\end{cor}

\prettyref{thm:limitSch} is similar in spirit to the results and
constructions involved in the first author's paper \cite{DocGro}, where
it is shown that the homotopy groups of a related space
$\Hom_*(T,G)$ can be determined by certain graph theoretic closed
paths in a (pointed) graph~$G^T$. In this context the role of a closed path
(a circle) is played by $C'_m$, a cycle of length $m$ with loops on all the vertices (the looped 1-skeleton of a
triangulated circle).  One constructs a graph~$\Omega G$ which parameterizes graph homomorphisms $C^\prime_m \rightarrow G$ from cycles of arbitrary length.  In particular it is shown that the space of (pointed) maps from a circle into $\Hom_*(K_2,G)$ can be recovered as path components of $\Hom_*(K_2, \Omega G)$, which in turn can be approximated by the spaces $\Hom_*(K_2 \times C^\prime_m, G)$.

\subsection{New results}\label{sec:results}

In this paper, we generalize the constructions discussed above by
showing that one can take graphs obtained as the 1-skeleton of
topologically `desirable' spaces and apply them directly to $\Hom$
complexes.  In this way we unify existing results regarding $\Hom$
complexes as well as provide new theorems and constructions.  In our
applications, the spaces of interest will primarily be spheres with a
$({\mathbb Z}_2 \times {\mathbb Z}_2)$-action given by the
antipodal/reflection maps.  One obtains a graph by taking a looped
1-skeleton of a triangulation of the space, which can then be utilized
in the context of the $\Hom$ complex.  More generally, if $P$ is a
poset we obtain a graph $P^1$ (see \prettyref{def:1}) whose vertices are the atoms
of $P$ with adjacency $x \sim y$ if and only if there exists a $z$
such that $z \geq x$ and $z \geq y$.  In the case that $P$ is the face
poset of a triangulation of the 1-sphere, we recover the graphs $C^\prime_m$ discussed
above.  The graphs obtained as $P^1$ have loops on every vertex and
hence do not admit homomorphisms to graphs with finite chromatic
number.  However, our results show the ${\mathbb
  Z}_2$-product of such graphs with a given ${\mathbb Z}_2$-graph~$T$
(for example $T = K_2$) interacts well with the relevant
$\Hom$ complexes.  Our two main structural results are the following.
All necessary definitions are provided in the
next section.

\begin{thm}\label{thm:firstentry}
Let $\Gamma$ be a finite group, and suppose $P$ is a poset with a left $\Gamma$-action and $T$ is a graph with a right $\Gamma$-action.  Then for any graph~$G$ there is a natural homotopy equivalence
\[\Poset_{\Gamma}\bigl(P, \Hom(T,G)\bigr) \simeq \Hom\bigl(T
\times_{\Gamma} P^1, G\bigr).\]
\end{thm}

\begin{thm}\label{thm:secondentry}
Let $\Gamma$ be a finite group, $G$ a graph, $T$ a graph with a right
$\Gamma$-action, and $P$ a poset with a free left $\Gamma$-action.  Let $d$ be the minimal diameter of a spanning tree in $G$.
If the induced action on the graph $P^1$ is
$\left(\max\set{5,d+2}\right)$-discontinuous
we have a natural homotopy equivalence
\[\real{\Hom(G, T\times_\Gamma P^1)} \simeq \real{\Hom(G, T)}\times_\Gamma
\real{\Hom(G, P^1)}.\]
\end{thm}

\begin{rem}
For the notion of $m$-discontinuity see \prettyref{def:discont}.  If
the action on~$P$ is free, then by \prettyref{lem:discont-subdiv} the action on the
poset $\Chain^{k} P$ is $2^k$-discontinuous.  We will usually apply \prettyref{thm:secondentry} in situations where we have $\real{\Hom\bigl(G, (\Chain^k P)^1 \bigr)}\homot\real{\Chain^k P}\homeo \real P$.
For example, \prettyref{cor:ldismantlable} gives sufficient conditions.
\end{rem}

\begin{rem}\label{rem:homotopy}
When we refer to a homotopy equivalence between posets $P$ and~$Q$ as
in \prettyref{thm:firstentry}, we will mean a homotopy equivalence
$\real P\to\real Q$ such that the homotopy equivalence as well as its
homotopy inverse are induced by poset maps.  In
\prettyref{sec:enrich} we suggest a category~$\mathcal P_0$ such
that $P\homot Q$ could refer to an isomorphism in that category.
\end{rem}

Loosely speaking \prettyref{thm:firstentry} says that if $P$ is a
topological space with a $\Gamma$-action (in the form of its face
poset), one can describe the space of $\Gamma$-equivariant maps from
$P$ into the complex $\Hom(T,G)$ in terms of the space of graph
homomorphisms from the graph $T \times_\Gamma P^1$ into $G$.  This provides a basic link between equivariant topology and the existence of graph
homomorphisms and explains our interest in the graphs $T \times_\Gamma
P^1$.  \prettyref{thm:secondentry} then allows us to study the space of
graph homomomorphisms to such a graph, and also describes the
(equivariant) topology of certain fiber bundles involving the spaces
$\Hom(G,\qm)$ in terms of $\Hom$ complexes from $G$ into these twisted graph
products.

\prettyref{thm:firstentry} and \prettyref{thm:secondentry} lead us to a number of applications.  We provide the details for these in \prettyref{sec:testgraphs}, \prettyref{sec:furtherapps} and \prettyref{sec:univ}, but wish to briefly describe the ideas here.  Our foremost application of \prettyref{thm:firstentry} in \prettyref{sec:testgraphs} will be related to the construction of new \textit{test graphs}.  Following \cite{Kchr}, we say that a graph~$T$ is a
\textit{homotopy test graph} if for all graphs $G$ we have
\[\chi(G) \geq \conn\bigl(\Hom(T,G)\bigr) + \chi(T) + 1.\]
Here $\conn(X)$ is the topological connectivity of the space $X$.  A graph~$T$ with a ${\mathbb Z}_2$-action that flips an edge is called a \textit{Stiefel-Whitney test graph} if for all $G$ with $\Hom(T,G) \neq \emptyset$ we have
\[\chi(G) \geq \Ht_{{\mathbb Z}_2} \bigl(\Hom(T,G)\bigr) + \chi(T).\]
In this language the results of Lov\'{a}sz, Babson-Kozlov, and the second author say that the edge $K_2$ and the odd cycles $C_{2r+1}$ are Stiefel-Whitney test graphs.  We point out that the constant $\chi(T)$ is best possible since if we take $T = G$ we get that $\Hom(G,G)$ non-empty and hence ($-1$)-connected.

To build new test graphs, we take $P$ in \prettyref{thm:firstentry} to be the face poset of a (properly subdivided) $k$-sphere.  We then recover the space of equivariant maps from the $k$-sphere ${\mathbb S}^k$ into the complex $\Hom(T,G)$ as a colimit of the complexes $\Hom(T \times_{{\mathbb Z}_2} P^1, G)$
(details are below). As a consequence we see that if $T$ is a Stiefel-Whitney test graph with certain additional properties (satisfied for example by $K_2$ and $C_{2r+1}$), then so is $T \times_{{\mathbb Z}_2} P^1$; in addition, certain topological invariants (e.g. connectivity) of $\Hom(T,G)$ are closely related to those of $\Hom(T \times_{{\mathbb Z}_2} P^1, G)$.  As discussed above, the odd cycles $C_{2r+1}$ form a directed family of test graphs with the property that the topology of $\Hom(C_{2r+1}, G)$ can be related to that of $\Hom(K_2, G)$.  We view our results as a generalization of this phenomenon, with the graphs $K_2 \times_{{\mathbb Z}_2} P^1$ serving as `higher-dimensional' analogues of the odd cycles.

In particular this gives us a general inductive procedure for constructing new test graphs of arbitrary chromatic number: one starts with a test graph $T$ and repeatedly applies the construction $\qm \times_{{\mathbb Z}_2} P^1$, for $P$ the face poset of a properly divided $k$-sphere.  In this paper we focus our attention on two new infinite families of test graphs, each parameterized by a pair $(k,m)$.  The parameter $k$ is related to the chromatic number of the test graph, whereas $m$ is a measure of its `fineness'.  Details are provided in \prettyref{sec:testgraphs} but we wish to give a brief description of these families here.

The collection of \emph{spherical graphs}, denoted $S_{k,m}$, are obtained as follows.  We let $X^k_m$ denote the $m$th barycentric subdivision of the boundary of the regular $(k+1)$-dimensional cross polytope, and let $F(X^k_m)$ denote its face poset.  We then define $S_{k,m} \deq K_2 \times_{{\mathbb Z}_2} \bigl(F(X^k_m)\bigr)^1$ to be the graph obtained by taking the twisted product of $K_2$ with the reflexive graph of the 1-skeleton of $X^k_m$.  For each $k$, the map of posets $F(X^k_{m+1}) \rightarrow F(X^k_{m})$ induces graph homomorphisms $S_{k,m+1} \rightarrow S_{k,m}$.  We will see in \prettyref{sec:testgraphs} that for each $k \geq 0$ the graphs $S_{k,m}$ are Stiefel-Whitney test graphs with chromatic number $k+3$, and in addition they will play a role in a generalized notion of \emph{homomorphism duality} discussed in \prettyref{sec:duality}.

We obtain the \emph{twisted toroidal graphs}, denoted $T_{k,m}$, by repeatedly taking twisted products with graphs obtained from subdivisions of a circle.  In this case the relevant posets have a simple combinatorial description, and to emphasize this we introduce some new notation.  For $m \geq 2$ we let ${\mathcal C}_{2m}$ denote the face poset of a $2m$-gon; it will be these posets that we use for $P$ in \prettyref{thm:firstentry}.  The graphs $\Cpm$ form a linear direct system which, for each $k \geq 0$, again leads to a collection of Stiefel-Whitney test graphs with chromatic number $k+3$.  The graphs $T_{k,m}$ have the property that their maximum degree is independent of $m$ (analogous to the fact that all odd cycles have maximum degree 2), and this leads to partial progress towards a conjecture of Lov\'{a}sz regarding bounds on chromatic number in terms of connectivity of test graphs of bounded degree.

We conclude \prettyref{sec:testgraphs} with a study of a family of graphs obtained from the \emph{generalized Mycielski} construction.  In particular we use our methods to show that the graphs obtained this way provide another family of test graphs with arbitrarily large chromatic number.

In \prettyref{sec:univ} we discuss our primary application of \prettyref{thm:secondentry}, namely the notion of $S_n$-universality for $\Hom(K_n,\qm)$ complexes.
In \cite{Cs05} Csorba shows that any finite simplicial complex with a free ${\mathbb
  Z}_2$-action can be approximated up to ${\mathbb Z}_2$-homotopy
equivalence as a complex $\Hom(K_2,G)$ for an appropriate choice of
graph~$G$ (see also \cite{Ziv05} for an independent proof of this).
We describe how his construction fits into our set-up, and we generalize
his result to establish the following.

\begin{thm}\label{thm:univintro}
[\prettyref{thm:univn}]
Let $X$ be a finite simplicial complex with a free $S_n$-action for $n \geq 2$.
Then there exists a loopless graph $G$ and $S_n$-homotopy equivalence
\[|\Hom(K_n, G)| \simeq_{S_n} |X|,\]
where $S_n$ acts on the left hand side as the automorphism group of~$K_n$.
\end{thm}

\noindent
In our set-up the desired graph is constructed as $G \deq K_n \times_{S_n} P^1$, where
$P$ is a the face poset of the given complex $X$, sufficiently
subdivided.  When $n=2$, we show how this recovers the construction of
Csorba.

The rest of the paper is organized as follows.  In \prettyref{sec:defs} we review relevant definitions and notation.  In \prettyref{sec:testgraphs} we describe explicitly our methods for the construction of new test graphs, and in particular the spherical and twisted toroidal graphs mentioned above.  In \prettyref{sec:furtherapps} we discuss other applications of these results in the context of homomomorphism duality and graph-theoretic interpretations of $\Hom$ complexes, as well as the $S_n$-universality of $\Hom$ complexes.  \prettyref{sec:proofs} is devoted to the proofs of the main theorems as well as some technical lemmas.  We conclude in \prettyref{sec:enrich} with some comments regarding the categorical content of our constructions, in particular in the context of \emph{enriched} category theory.

\noindent \textbf{Acknowledgments.}  The authors would like to thank
Eric Babson for useful conversations, as well as the anonymous referee for helpful comments and corrections.  The first author was supported by the Deutscher Akademischer Austausch Dienst (DAAD) and by a postdoctoral fellowship from the Alexander von Humboldt Foundation.  Both authors would like to thank the organizers of the MSRI Program on Computational Applications of Algebraic Topology in Fall 2006, where many of these ideas were developed.

\section{Definitions and conventions}\label{sec:defs}
In this section we provide a brief overview of some notions from the theory of graphs, $\Hom$ complexes, and general ${\mathbb Z}_2$-spaces.  We refer to ~\cite{BKcom} and ~\cite{Kchr} for a more thorough introduction to the subject.

For us a \textit{graph} $G$ is a finite set of vertices $V(G)$ with a
symmetric adjacency relation $E(G) \subset V(G) \times V(G)$;
hence our graphs are undirected without multiple edges, but
possibly with loops.  If $v$ and $w$ are vertices of $G$ such that
$(v,w) \in E(G)$ then we will often say that $v$ and $w$ are
\textit{adjacent} and denote this $v \sim w$.  A \textit{graph homomorphism}
$f:G \rightarrow H$ is a vertex set map $V(G) \rightarrow V(H)$ that
preserves adjacency: if $v \sim w$ in $G$ then $f(v) \sim f(w)$ in
$H$.  The \emph{complete graph} $K_n$ has vertices $\{1, 2, \dots, n\}$ and all possible non-loop edges.  Given a graph $G$, we define $\chi(G)$, the \emph{chromatic number} of $G$, to be the minimum $n$ such that there exists a graph
homomorphism $G \rightarrow K_n$.

\begin{defn}
Let $G$ and $H$ be graphs.  The categorical \emph{product} $G \times H$ is
the graph with vertex set $V(G) \times V(H)$ and with adjacency
given by $(v,w) \sim (v^\prime, w^\prime)$ if $v \sim v^\prime$ and
$w \sim w^\prime$.  The \textit{exponential} graph~$H^G$ is the graph on the
vertex set $V(H)^{V(G)}$ with adjacency $f \sim f^\prime$ if $f(v)
\sim f^\prime(w)$ for all $v \sim w$ in $G$.
\end{defn}

A graph is called \emph{reflexive} if the adjacency relation is
reflexive, i.e. if all of the vertices have loops. A graph is called \emph{loopless} if there are no loops on any of the vertices.  The graph~${\bf
1}$ is defined to be the (reflexive) graph with a single looped
vertex.  Note that there are natural isomorphisms $G \times {\bf 1} \rightarrow G$ and $G \rightarrow G^{\bf 1}$.

\begin{defn}
Let $G$ be a graph with a given equivalence relation $R$ on its vertices.  The \textit{quotient graph} $G/R$ is the graph with vertices $V(G)/R$ and with adjacency $[v] \sim [w]$ if there exists $v^\prime \in [v]$ and $w^\prime \in [w]$ such that $v^\prime \sim w^\prime$ in $G$.  \
\end{defn}

For our applications, the equivalence relation will most often be given by the orbits of some group action.  Recall that if $\Gamma$ is a group, and $X$ and $Y$ are spaces with (respectively) a right and a left $\Gamma$-action, then $\Gamma$ acts diagonally on the product $X \times Y$ according to $\gamma \cdot (x,y) \deq (x \gamma^{-1}, \gamma y)$.
The space $X \times_{\Gamma} Y$ is then defined to be the orbit space under this action, so that $X \times_{\Gamma} Y \deq (X \times Y) /
\sim$, where $(x \gamma,y) \sim (x, \gamma y)$.  Similarly, we have the following construction for graphs.

\begin{defn}
Let $G$ be a graph with a left $\Gamma$-action and $H$ a graph with a right $\Gamma$-action.  Define $G \times_{\Gamma} H$ to be the graph with vertices given by the orbits of the diagonal $\Gamma$-action on $G \times H$, with adjacency given
by $[(g,h)] \sim [(g^\prime, h^\prime)]$ if there exists
representatives in $G \times H$ with $(g,g^\prime) \sim (h,
h^\prime)$.
\end{defn}

\begin{defn}\label{def:discont}
Let $G$ be a graph with a left $\Gamma$-action for some group
$\Gamma$.  For an integer $d>0$, we say that the action is
\emph{$d$-discontinuous} if for each vertex $v \in G$, the
neighborhood $N_{d-1}(v)$ of radius~${d-1}$ around~$v$ has the property that
$\gamma v \notin N_{d-1}(v)$ for all nonidentity $\gamma \in
\Gamma$.
\end{defn}

We next come to the construction of the $\Hom$ complex.  We point out that our definition is slightly different from the one given in ~\cite{BKcom} in the sense that the $\Hom$ complex we define here is the face poset of the polyhedral complex given in ~\cite{BKcom}.  Since the geometric realization of the face poset of a regular cell complex is homeomorphic to the complex itself, the underlying spaces of both $\Hom$ complexes are the same.

\begin{defn}
Let $G$ and $H$ be graphs.  We define $\Hom(G,H)$ to be the poset whose elements are all set maps
$\alpha:V(G) \rightarrow 2^{V(H)} \backslash \{\emptyset\}$ with the
condition that if $g \sim g^\prime$ in $G$ then $h \sim h^\prime$
for all $h \in \alpha(g)$ and $h^\prime \in \alpha(h)$.  The partial
order is given by $\alpha \leq \beta$ if $\alpha(g) \subset
\alpha(g^\prime)$ for all $g \in V(G)$.
\end{defn}

For any graph $T$, $\Hom(T,\qm)$ is a functor from graphs to posets.  We will also need the following construction as a way
to obtain a (reflexive) graph from a poset.

\begin{defn}\label{def:1}
Let $P$ be a poset.  We define $P^1$ to be the reflexive graph with
vertices given by the atoms of $P$, and with adjacency given by $x
\sim y$ if there exists $z \in P$ with $z \geq x$ and $z \geq y$.
\end{defn}

Note that the atoms of $\Hom(G,H)$ are precisely the homomorphisms
$f:G \rightarrow H$.  We will sometimes refer to (arbitrary) elements
of $\Hom(G,H)$ as \textit{multihomomorphisms}.  The poset $\Hom(G,H)$
is ranked according to $\rank(\alpha) = \displaystyle\sum_{v \in V(G)}
\bigl( |\alpha(v)| - 1\bigr)$, for $\alpha \in \Hom(G,H)$.  We will
often speak about topological properties of the $\Hom$ complexes, and
in this context we will be referring to the (geometric realization of
the) order complex of the poset $\Hom(G,H)$.  We will use the notation
$|\Hom(G,H)|$ to emphasize the distinction but will also use simply
$\Hom(G,H)$ when the context is clear.

The $\Hom$ complex is functorial in both entries, and in particular $\Hom(T,G)$ carries an action by $\Aut(T)$, the automorphism group of the graph $T$.  If $T$ has an involution that flips an edge, this then induces a \emph{free} ${\mathbb Z}_2$-action on the space $\Hom(T,G)$ for any loopless graph $G$ (see for instance \cite{Kchr}).  The examples discussed in the introduction arise from taking $T$ to be an edge $K_2$ with the nonidentity involution, or to be an odd cycle $C_{2r+1}$ on the vertices $\{0, \dots, 2r\}$ with the reflection given by $i \mapsto 2-i \pmod {2r+1}$.

If $X$ is a space with a (free) ${\mathbb Z}_2$-action, there are several invariants used to measure the complexity of the action.  We collect some of these notions in the next definition.

\begin{defn}  Let $X$ be a space with a free ${\mathbb Z}_2$-action, and let ${\mathbb S}^n$ denote the $n$-sphere endowed with the antipodal action.  We define the \emph{index} and \emph{coindex} of $X$ as follows:
\begin{align*}
\ind X &\deq \min\{m: X \rightarrow_{\Z_2} {\mathbb S}^m \}\\
\coind X &\deq \max\{n: {\mathbb S}^n \rightarrow_{\Z_2} X\},
\end{align*}
where $\rightarrow_{\Z_2}$ denotes a ${\mathbb Z}_2$-equivariant map.  The \emph{height} of $X$, denoted $\Ht_{\Z_2} X$, is defined to be the highest nonvanishing power of the first Stiefel-Whitney class of the bundle ${\mathbb R} \times_{{\mathbb Z}_2} X \rightarrow X/{\mathbb Z}_2$.
\end{defn}
\noindent
We refer to \cite{M03} for further discussion of these invariants, and especially their use in combinatorial applications.  One can check that if $X$ a free ${\mathbb Z}_2$-space, these values are related in the following way:
\[\conn X + 1 \leq \coind X \leq \Ht_{{\mathbb Z}_2} X \leq \ind X.\]

Finally, we collect a couple notions from the theory of posets.

\begin{defn}
Let $P$ be a poset.  Define $\Chain P$ to be the poset whose elements are the nonempty chains $x_1 \leq \cdots \leq x_n$ of $P$, with the relation given by containment.
\end{defn}

\begin{defn}
Let $P$ and $Q$ be posets.  Then $\Poset(P,Q)$ is the poset of all
order preserving maps $f\colon P \rightarrow Q$, with the relation $f \leq g$ if
$f(x) \leq g(x)$ for all $x\in P$.  If $P$ and $Q$ are both equipped with actions by some group $\Gamma$, we let $\Poset_\Gamma(P,Q)$ denote the subposet of $\Poset(P,Q)$ consisting of all equivariant poset maps.
\end{defn}

\section{Constructing new test graphs}\label{sec:testgraphs}
In this section we provide details regarding our primary application of \prettyref{thm:firstentry}, namely the construction of new test graphs for topological bounds on
chromatic number.  We begin with a brief discussion regarding the definition and history of such graphs, as well as our general approach to their construction.

Recall from \cite{Kchr} that a graph~$T$
is called a \textit{homotopy test graph} if for every graph~$G$, we
have the following inequality:
\[\chi(G) > \chi(T) + \conn\bigl(\Hom(T,G)\bigr).\]
The results of Lov\'{a}sz and Babson, Kozlov imply that the complete graphs $K_n$, $n\ge2$, and
the odd cycles $C_{2r+1}$ are homotopy test graphs.  For some time
it was an open question whether \textit{all} graphs were homotopy
test graphs, but Hoory and Linial showed that this was not the case
in \cite{HLcounter} by constructing a graph~$H$ with $\chi(H) = 5$
such that $\Hom(H, K_5)$ is connected.  In fact there are very few
graphs that are known to be test graphs (see \cite{Schspace} and
\cite{Zivcom} for some discussion regarding this).

Now suppose $T$ is graph with a ${\mathbb Z}_2$-action that flips an
edge.  Also from \cite{Kchr}, we say that a graph~$T$ is a
\textit{Stiefel-Whitney test graph} if for all $G$ with $\Hom(T,G) \neq \emptyset$ we have
\begin{equation}\label{eq:testgraph}
\chi(G) \geq \chi(T) + \Ht_{{\mathbb Z}_2}\bigl(\Hom(T,G)\bigr).
\end{equation}
\noindent
We point out that it is enough to restrict graphs in the second coordinate to the set of all complete graphs $K_n$ with $n \geq \chi(T)$.  Indeed, we have $G \rightarrow K_{\chi(G)}$ and $\Ht_{{\mathbb Z}_2}\bigl(\Hom(T,G)\bigr) \leq \Ht_{{\mathbb Z}_2}\bigl(\Hom(T,H)\bigr)$ whenever $G \rightarrow H$.  Hence for any $G$ we get
\[\Ht_{{\mathbb Z}_2}\bigl(\Hom(T,G)\bigr) \leq \Ht_{{\mathbb Z}_2}\bigl(\Hom(T,K_{\chi(G)})\bigr) \leq \chi(K_{\chi(G)}) - \chi(T) = \chi(G) - \chi(T).\]
Also, in \cite{Kchr} Kozlov insists on equality in the formulation involving complete graphs, but for our purposes the inequality will suffice.  Note that the existence of an equivariant coloring $T\to_{\Z_2}K_{\chi(T)}$ (as defined in \prettyref{lem:colorTkm}) will in fact imply such an equality, since such a coloring induces an equivariant map
$\Sphere^{n-\chi(T)}\homeo_{\Z_2}\Hom(K_2, K_{n-\chi(T)+2})
\to_{\Z_2}\Hom(K_{\chi(T)}, K_n)\to_{\Z_2}\Hom(T,K_n)$ which implies
\[\htt\left(\Hom(T,K_n)\right)\ge
\coind_{\Z_2}\left(\Hom(T,K_n)\right)\ge n-\chi(T).\]

Also note that every Stiefel-Whitney test graph (in our sense) is also a homotopy test graph since $\conn(X) + 1 \leq \Ht_{{\mathbb Z}_2}(X)$ for a ${\mathbb Z}_2$-space $X$.  Hence if $T$ is a Stiefel-Whitney test graph and $G$ is a graph with $\Hom(T,G) \neq \emptyset$ we have
\[\chi(T) \leq \chi(G) - \Ht_{{\mathbb Z}_2}\bigl(\Hom(T,G)\bigr) \leq \chi(G) - \conn \bigl(\Hom(T,G)\bigr) - 1.\]

We next describe our method for constructing new Stiefel-Whitney test
graphs.  As mentioned above, we will obtain these graphs by possibly
repeated applications of the $\qm \times_{{\mathbb Z}_2} P^1$
construction, where $P$ is a symmetric triangulation of ${\mathbb
  S}^k_b$.  The basic result which makes this possible is the
following.

\begin{prop}\label{prop:main}
Let $T$ be a graph with a right ${\mathbb Z}_2$-action and let $G$ be a graph.  For $k\ge0$, let $P$ be the face poset of a
$(\Z_2\times\Z_2)$-triangulation of~$\Sphere^k_b$.
If $\Hom(T \times_{\Z_2} P^1, G) \neq \emptyset$ then we have
\[\Ht_{{\mathbb Z}_2}\bigl(\Hom(T \times_{{\mathbb Z_2}} P^1,G)\bigr) + k \leq \Ht_{{\mathbb Z}_2}
\bigl(\Hom(T,G)\bigr).\]
Hence if $T$ is a Stiefel-Whitney test graph, we have $\chi(T
\times_{\Z_2} P^1)\ge \chi(T) +k$, and for any graph~$G$ such that $\Hom(T
\times_{\Z_2} P^1, G)\ne\emptyset$ we get
\[\Ht_{{\mathbb Z}_2}\bigl(\Hom(T\times_{\Z_2} P^1, G)\bigr)+k+\chi(T) \leq \Ht_{{\mathbb Z}_2} \bigl(\Hom(T,G)\bigr) +\chi(T)\le\chi(G).\]
\end{prop}

\begin{proof}
\prettyref{thm:firstentry} yields an equivalence and hence in particular
(recalling our convention laid down in \prettyref{rem:homotopy}) a poset map
\[\Hom(T \times_{\Z_2} P^1, G)\to_{\Z_2}
\Poset_{\Z_2}(P,\Hom(T, G)).
\]
Since $\Sphere^k_b\homeo_{\Z_2\times\Z_2}\real P$, we obtain a continuous map
\[
\Sphere^k_b\times_{\Z_2}\real{\Hom(T\times_{\Z_2} P^1, G)}
\to_{\Z_2}\real{\Hom(T,G)}.\]
The first inequality now follows from \prettyref{prop:Sk}.
If $T$ is assumed to be a Stiefel-Whitney test graph
then the last inequality follows
from \prettyref{eq:testgraph}.
Finally, setting $G=T\times_{\Z_2}P^1$ we obtain
\[0 + k + \chi(T) \leq \Ht_{{\mathbb Z}_2}\bigl(\Hom(T \times_{{\mathbb Z_2}} P^1,T \times_{{\mathbb Z_2}} P^1)\bigr) + k + \chi(T) \leq \chi(T \times_{{\mathbb Z_2}} P^1),\]
\noindent
so that $\chi(T \times_{{\mathbb Z_2}} P^1) \geq k + \chi(T)$.
\end{proof}

In \prettyref{sec:spherical} and \prettyref{sec:toroidal} we use \prettyref{prop:main} to construct infinite families of test graphs, with special attention paid to the \emph{spherical} and \emph{twisted toroidal} graphs.

As the reader may have noticed in the proof of \prettyref{prop:main},
the full strength of \prettyref{thm:firstentry} is not needed to
establish the desired bounds on $\Ht_{{\mathbb Z}_2} \Hom(T\times_{\Z_2}P^1,G)$.
In fact, to establish upper bounds
on $\htt\Hom(T', G)$ for some $\Z_2$-graph~$T'$,
it is enough to construct $\Z_2$-maps
\[\Sphere^k_b\times_{\Z_2}\real{\Hom(T',G)}\to_{\Z_2}\real{\Hom(T,G)}\]
where $T$ is a $\Z_2$-graph (often $T=K_2$) for which upper bounds on
$\htt\Hom(T,G)$ are known.  \prettyref{prop:Sk} then yields
\[\htt\Hom(T', G)+k\le
\htt(\Sphere^k_b\times_{\Z_2}\real{\Hom(T',G)})
\le\htt\Hom(T,G).\]

This method was used in \cite{Schspace} in the context of odd cycles
in the first coordinate of the $\Hom$ complexes.  In \prettyref{sec:mycielski} we will employ this method to show that the
\emph{generalized Mycielski} graphs provide another family of test
graphs with arbitrarily high chromatic number.  The full strength of \prettyref{thm:firstentry} will be used in the context of spherical and twisted toroidal
graphs to establish the existence of graph homomorphisms from these graphs, as in \prettyref{prop:coind} and \prettyref{cor:mapsfrom}.

\subsection{Spherical graphs}\label{sec:spherical}
The most natural application of \prettyref{prop:main} comes from setting $T \deq K_2$ to obtain the graph $K_2 \times_{\Z_2}P^1$, where $P$ is the face poset of a (${\mathbb Z}_2 \times {\mathbb Z}_2$)-triangulation of ${\mathbb S}^k_b$.  As we have seen, the graph $K_2$ is a Stiefel-Whitney test graph, and hence from \prettyref{prop:main} we get
\begin{equation}\label{eq:snk-test}
\htt\Hom(K_2\times_{\Z_2} P^1, G)+k+2\le\htt\Hom(K_2,G)+2\le\chi(G).
\end{equation}

It now follows that the graph $K_2\times_{\Z_2} P^1$ is a test graph if its chromatic number
is $k+2$ and \prettyref{prop:main} also tells us that
$\chi(K_2\times_{\Z_2} P^1)\ge k+2$.
On the other hand, $K_2\times_{\Z_2} P^1$
need not be $(k+2)$-colorable in general (for example a
triangulation of $\Sphere^1_b$ as a $4$-gon yields $K_4$), but it is if
the triangulation $P$ is fine enough.  We describe concrete colorings in
\prettyref{prop:col-n-compl}, but it also follows more abstractly
from the following result,
which also gives a graph theoretical interpretation of the
coindex of the space $\Hom(T,G)$,
 for $T$ a graph with an involution that flips an edge.

\begin{prop}\label{prop:coind}
Let $T$ be a graph with an involution that flips an edge.  For each $k$,
suppose $\{X^k_m\}_{m \geq 0}$ is a sequence of symmetrically (with
respect to the antipodal action) triangulated $k$-spheres such the
maximal diameter of a simplex of $X^k_m$ tends to zero when $m$ tends
to infinity (e.g. take $X^k_m$ to be the $m$th barycentric subdivision
of the boundary of the regular $k+1$-dimensional cross polytope).  Let
$G^k_m \deq \bigl(F(X^k_m)\bigr)^1$ be the reflexive graph given by the
1-skeleton of $X^k_m$.  Then
\begin{align*}
\coind_{\Z_2}\Hom(T, H)
&=\max\set{k\colon\text{There are $m\ge0$ and $T \times_{{\mathbb Z}_2} G^k_m\to H$}}
\\
&=\max\set{k\colon\text{For almost all $m$ there is
$T \times_{{\mathbb Z}_2} G^k_m\to H$}}.
\end{align*}
\end{prop}

\begin{proof}
If there is a graph homomorphism $T \times_{{\mathbb Z}_2} G^k_m\to H$
then by \prettyref{thm:firstentry}
\[\emptyset\ne\Hom(T \times_{{\mathbb Z}_2} G^k_m,H)\homot
\Poset_{\Z_2}(F(X^k_m), \Hom(T, H))\] and hence there is an equivariant
map
\[\Sphere^k\homeo_{\Z_2}\real{F(X^k_m)}\to_{\Z_2}\real{\Hom(T, H)}\]
which means that $\coind_{\Z_2}\Hom(T, H)\ge k$.

On the other hand, if $\coind_{\Z_2}\Hom(T, H)\ge k$, then by
simplicial approximation there is an equivariant simplicial map from
$X^k_m$ to the barycentric subdivision of~$\Hom(T, H)$, whenever the
simplices of $X^k_m$ are small enough, and therefore for almost
all~$m$.  Such a simplicial map induces a poset map
$F(X^k_m)\to\Hom(T, H)$ by sending a simplex to the maximum of the
images of its vertices.  This shows that $\emptyset\ne
\Poset_{\Z_2}(F(X^k_m), \Hom(T, H))\homot\Hom(T \times_{{\mathbb Z}_2} G^k_m,H)$.
\end{proof}

\begin{cor}\label{cor:Snk-test}
Let $k\ge0$ and $G^k_m$ be as above.  Then there is an $M\ge0$ such that
for all $m\ge M$ the equality $\chi(K_2\times_{\Z_2} G^k_m)=k+2$ holds and
hence $K_2\times_{\Z_2} G^k_m$ is a Stiefel-Whitney test graph.
\end{cor}

\begin{proof}
We only have to show $\lim_k \chi(K_2\times_{\Z_2}
G^n_k)\le n+2$, and this follows from the preceding proposition and
$\coind_{\Z_2}\Hom(K_2, K_{n+2})=\coind_{\Z_2}\Sphere^n=n$.
\end{proof}

\begin{prop}\label{prop:col-n-compl}
Let $X$ be the second barycentric subdivision of a regular
$n$-dimensional cell complex with a free cellular $\Z_2$-action, and let $F(X)$ denote its face poset.  Then
$\chi(K_2\times_{\Z_2} F(X)^1)\le n+2$.
\end{prop}

\begin{proof}
Let $P$ be the face poset of the complex of which $X$ is the second
barycentric subdivision.  Then $F(X)\isom\Chain^2 P$ and we identify
the vertices of $F(X)^1$ with the elements of $\Chain P$.  We now
choose one element of each orbit of the free $\Z_2$-action on~$P$ and
call the set of all chosen elements~$S$.  For a face $p\in P$ we
denote its dimension by $d(p)$.  We now define a function
\begin{align*}
\phi\colon \Chain P&\to V(K_{n+2})=\set{1,\dots,n+2},\\
c&\mapsto\begin{cases}
\max\set{d(p)+1\colon p\in S\intersect c},&S\intersect c\ne\emptyset,\\
n+2,&S\intersect c=\emptyset.
\end{cases}
\end{align*}
For $c,d\in\Chain P$ with $\tau c\le d$ we show that
$\phi(c)\ne\phi(d)$.  Assume that $\phi(c)=n+2$.  Then $\tau c\subset
S$ and $\emptyset\ne \tau c\intersect S\subset d\intersect S$ and
hence $\phi(d)\ne n+2$.  Now let $\phi(c)<n+2$.  Then there is a $p\in
S\intersect c$ with $d(p)+1=\phi(c)$.  Since $\tau p\in d\wo S$,
we have $\phi(d)\ne d(\tau p)+1=d(p)+1$.
This shows that the map
\begin{align*}
V(K_2\times_{\Z_2} F(X)^1)&\to V(K_{n+2}),\\
[(1,c)]&\mapsto \phi(c)
\end{align*}
is a graph homomorphism, since $[(1,c)]$ is a neighbor
of~$[(1,d)]=[(2,\tau d)]$ if and only if $\tau c\le d$ or $\tau d\le
c$.
\end{proof}

\prettyref{prop:coind} gives us several candidates for families of test graphs (depending on which triangulations of spheres that we choose).  We wish to fix the following as the family of \emph{spherical graphs}.

\begin{defn}[Spherical graphs] \label{def:Snk}
Let $m,k\ge0$ and $X^k_m$ to be the $m$th barycentric subdivision
of the boundary of the regular $k+1$-dimensional cross polytope.
Then we define a loopless graph $S_{k,m}$ with a right $\Z_2$-action
by
\[S_{k,m}\deq K_2\times_{\Z_2}\bigl(F(X^k_m)\bigr)^1.\]
\end{defn}

\begin{cor}\label{cor:chi-Snk}
For $k\ge0$ and $m\ge1$ the graphs $S_{k,m}$ are Stiefel-Whitney test graphs with $\chi(S_{k,m})=k+2$.
\end{cor}

\begin{proof}
Since the boundary of the $(k+1)$-dimensional cross polytope is the
barycentric subdivision of a free $\Z_2$-cell complex with two cells
in each dimension from $0$ to $k$, \prettyref{prop:col-n-compl} implies
$\chi(K_2 \times_{{\mathbb Z}_2} F(X^k_m)^1)\le n+2$.  We have already seen that
the rest is a consequence of~\prettyref{eq:snk-test}.
\end{proof}

\begin{rem}\label{rem:HomK2Skm}
One might ask if for example it would have been enough to know the
connectivity of $\Hom(K_2, S_{k,m})$ to establish $\chi(S_{k,m})\ge k+2$.  Indeed, for $m\ge1$ this is sufficient, since in this case
a Theorem of Csorba (here \prettyref{thm:univ2}, see the remark there)
applies and yields
$\real{\Hom(K_2,S_{k,m})}\homot_{\Z_2}\Sphere^k$.
\end{rem}

\subsection{Twisted toroidal graphs}\label{sec:toroidal}
In this section we let $P$ be a $({\mathbb Z}_2 \times {\mathbb Z}_2$)-triangulation of a 1-sphere, and consider the case of repeatedly applying the $\qm \times_{{\mathbb Z}_2} P^1$ construction.  In this case, we have a particularly simple description of these graphs in mind.  We define $\Cm$ to be the face poset of a $2m$-gon with vertex set
$\set{0,\dots,2m-1}$, with the antipodal left action given by $i \mapsto i + m$ (mod $2m$) and the
reflection right action by $i \mapsto 2m - 1 - i$ (mod $2m$).
This yields $\real{\Cm}\homeo_{{\mathbb Z}_2 \times {\mathbb Z}_2}
{\mathbb S}^1_b$.  In \prettyref{thm:firstentry} and \prettyref{thm:secondentry} we take a
quotient of the product of graphs in the context of the ${\mathbb
  Z}_2$-action on $P^1$.  In the case $P=\Cm$, we will want
to consider the graphs of the sort $T \times_{{\mathbb Z}_2}
\Cpm$ (see \prettyref{fig:K2C3} for the case of $T = K_2$ with the
nontrivial ${\mathbb Z}_2$-action).

\begin{figure}[!ht]
\begin{center}
\input{K2C3.pspdftex}
\caption{The graphs $K_2, \Cptx6$, and $T_{1,3} \deq K_2 \times_{{\mathbb
Z}_2} \Cptx6 $.
}\label{fig:K2C3}
\end{center}
\end{figure}

Iterating this construction gives the following family of `twisted toroidal' graphs.

\begin{defn}[Twisted toroidal graphs] \label{defn:Tgraphs}
For integers $k\ge1$, $m\ge2$ we define the graph
\begin{center}
$T_{k,m} \deq K_2 \underbrace{\times_{{\mathbb Z}_2} \Cpm
\times_{{\mathbb Z}_2} \cdots \times_{{\mathbb Z}_2}
\Cpm}_{k-\textrm{times}}$.
\end{center}
\end{defn}
The example in \prettyref{fig:K2C3} is $T_{1,3}$.  We point out that each
$T_{k,m}$ is a graph without loops, since $K_2$ has no loops.  As was the case with the spherical graphs in the previous section, \prettyref{prop:main} gives us the following.

\begin{prop}~\label{prop:mainprop1}
Let $T$ be a graph with a (right) ${\mathbb Z}_2$-action and let $G$
be a graph.  Assume that
$\Hom(T \times_{{\mathbb Z_2}}
\Cpm,G)\ne\emptyset$.
Then for $m \geq 2$ we have
\[\Ht_{{\mathbb Z}_2}\bigl(\Hom(T \times_{{\mathbb Z_2}}
\Cpm,G)\bigr) + 1 \leq \Ht_{{\mathbb Z}_2}
\bigl(\Hom(T,G)\bigr).\]
In particular if $\Hom(T_{k,m}, G) \neq \emptyset$ for some $k \geq 1$ and $m \geq 2$, we have
\[\Ht_{{\mathbb Z}_2}\bigl(\Hom(T_{k,m}, G)\bigr)+k+2\le \Ht_{{\mathbb Z}_2}\bigl(\Hom(K_2,G)\bigr)+2\le\chi(G).\]
\end{prop}

To show that the $T_{k,m}$ are indeed Stiefel-Whitney test graphs we have to show that they have the desired chromatic number.  To establish this we will show that under certain additional assumptions, the construction $T\mapsto T\times_{\Z_2}\Cpm$ raises the
chromatic number by exactly one.  These additional assumptions will be
fulfilled for example when we pass from $T_{k,m}$ to $T_{k+1,m}$.

\begin{lem}\label{lem:colorTkm}
Let $T$ be a graph with a right $\Z_2$-action and $n\ge0$, $m\ge3$.
If there is
an equivariant graph homomorphism
\[T\to_{\Z_2}K_{n+2},\]
where the ${\mathbb Z}_2$-action on $K_{n+2}$ is given by
exchanging the vertices $1$ and $2$, and leaving all other vertices
fixed, then there is an equivariant graph homomorphism
\[T \times_{{\mathbb Z}_2} \Cpm\to_{\Z_2} K_{n+3}\]
with the $\Z_2$-action on $K_{n+3}$ as that on~$K_{n+2}$.
\end{lem}

\begin{proof}
First note that for any $m \geq 3$ we have a $({\mathbb Z}_2 \times
{\mathbb Z}_2)$-equivariant homomorphism $\Cpm \rightarrow \Cptx6$
given by $0 \mapsto 0$, $i \mapsto 1$ for $0 < i < m-1$, $m-1 \mapsto
2$, $m \mapsto 3$, $j \mapsto 4$ for $m < j < 2m-1$, and $2m-1 \mapsto
5$.  The looped vertices of $K_3^{K_2}$ induce a 6-cycle and can be identified with $\Cptx6$, and hence we have a $({\mathbb Z}_2 \times {\mathbb
  Z}_2)$-homomorphism $\Cpm \rightarrow K_3^{K_2}$.  We also have a
$({\mathbb Z}_2 \times {\mathbb Z}_2)$ homomorphism $K_3^{K_2}
\rightarrow K_{n+3}^{K_{n+2}}$ given by extending any element of
$K_3^{K_2}$ to $i \mapsto i+1$ for any vertex $i > 2$ (this is
equivariant since both actions are trivial on these vertices).  Now,
assume that $n+2 = N$ and that we have an
equivariant coloring $T \rightarrow K_{n+2}$.  This gives us
\[\Cpm \rightarrow K_3^{K_2} \rightarrow K_{n+3}^{K_{n+2}} \rightarrow K_{n+3}^T,\]
a sequence of $({\mathbb Z}_2 \times {\mathbb Z}_2)$-equivariant homomorphisms and hence an equivariant coloring
\[T \times_{{\mathbb Z}_2} \Cpm \rightarrow_{{\mathbb Z}_2} K_{n+3}\]
\noindent
as desired.
\end{proof}

\begin{cor}
For all $k\ge0$ and $m\ge3$ we have $\chi(T_{k,m})\le k+2$.
\end{cor}

\begin{proof}
Induction on $k$.
\end{proof}

With this machinery in place we can state our main result of this section.

\begin{prop}\label{prop:Tkmtestgraph}
Let $k\ge0$, $m\ge2$.  Then $T_{k,m}$ is a Stiefel-Whitney test graph of chromatic number~$k+2$.  In particular we have,
\[\Ht_{{\mathbb Z}_2} \Hom(T_{k,m}, G) + k \leq
\Ht_{{\mathbb Z}_2} \Hom(K_2, G) \leq \chi(G) -2\]
for every graph~$G$ such that $ \Hom(T_{k,m}, G)\ne\emptyset$.
\end{prop}

\begin{proof}
We have
$\htt\Hom(T_{0,m},G)=\htt\Hom(K_2,G)\le\htt\Hom(K_2,K_{\chi(G)})=\chi(G)-2$
and, again by \prettyref{prop:mainprop1},
$\htt\Hom(T_{k,m},G)+1\le\htt\Hom(T_{k-1,m},G)$ for $k>0$ and
$\Hom(T_{k,m},G)\ne\emptyset$.
\end{proof}

The graphs $T_{k,m}$ also have interesting properties relating their maximum degree and odd girth.

\begin{prop}\label{prop:Tkmgraphs}
Let $k\ge1$ and $m=2r+1$, $r\ge1$.  Then $T_{k,m}$ has chromatic number~$k+2$, odd girth~$m$, and every vertex has
degree~$3^k$.
\end{prop}

\begin{proof}
The equality $\chi(T_{k,m})=k+2$ has already been established.

Every vertex of $T_{0,m}=K_2$ has
degree~$1$ and every vertex of~$\Cpm$ has degree~$3$, it follows
inductively that every vertex of $T_{k,m}$ has degree~$3^k$.

Consider a cycle of length less than~$m$
in~$T_{k+1,m}=T_{k,m}\times_{\Z_2}\Cpm$. This is contained in a
subgraph isomorphic to $T_{k,m}\times \mathcal P_{m-1}^1$, where $\mathcal
P_{m-1}$ is the face poset of a triangulation of an interval with
$m$~vertices.  Since projection yields a graph homomorphism $T_{k,m}\times
\mathcal P_{m-1}^1\to T_{k,m}$ we see that if $T_{k,m}$ contains no odd
cycle of length less than~$m$ then neither does $T_{k+1,m}$.  Since
$T_{0,m}$ contains no odd cycle at all, this proves that the odd girth
of~$T_{k,m}$ is at least~$m$.

The two looped vertices $r$ and $r+m$ of $\Cpm$ are exchanged by both
actions.  They yield a subgraph of $T\times_{\Z_2}\Cpm$ which is
$\Z_2$-isomorphic to~$T$.  Hence for $k\ge1$ the graph $T_{k,m}$
contains a copy of the graph $T_{1,m}$, which contains a cycle of
length $m$ (in \prettyref{fig:K2C3} consider $A0,B1,A2,B3=A0$).
Therefore the odd girth of $T_{k,m}$ is exactly~$m$.
\end{proof}

\begin{rem}
Another way to obtain graphs with large chromatic number and large odd girth
is via the generalized Mycielski construction, see
\prettyref{sec:mycielski}.
However, for these graphs the maximal
degree will go to infinity when the odd girth goes to infinity, while
the maximal degree of $T_{k,m}$ is independent of~$m$.
We will make use of this property again in \prettyref{sec:Lovconj}.
\end{rem}

In \prettyref{prop:coind} of the previous section, we saw how homomorphisms from spherical graphs were related to the coindex of $\Hom$ complexes.  Here we can obtain a similar inequality in the context of twisted toroidal graphs.

\begin{prop}\label{prop:mainprop2}
Let $T$ be a graph with a (right) ${\mathbb Z}_2$-action, let $G$
be a graph and $m \geq 2$.
If $\coind_{{\mathbb Z}_2} \bigl(\Hom(T,G)\bigr) \geq
1$, then
\begin{align*}
\lim_{m \to \infty} \coind_{{\mathbb Z}_2}\bigl(\Hom(T
\times_{{\mathbb Z}_2} \Cpm, G)\bigr) + 1 &\geq \coind_{{\mathbb
Z}_2} \bigl(\Hom(T,G)\bigr).
\end{align*}
\end{prop}

\begin{proof}
We suppose $k \geq 0$ and assume
$\coind_{{\mathbb Z}_2} \bigl(\Hom(T,G)\bigr) \geq k + 1$.  Since
${\mathbb S}^1_b \times_{{\mathbb Z}_2} {\mathbb S}^k$ is a
$(k+1)$-dimensional free ${\mathbb Z}_2$-space, there exists an
equivariant map ${\mathbb S}^1_b \times_{{\mathbb Z}_2} {\mathbb
S}^k \rightarrow_{{\mathbb Z}_2} |\Hom(T,G)|$.  We apply simplicial
approximation to this map and obtain, for some $m \gg 0$, a
${\mathbb Z}_2$-poset $P$ with $\real P\approx_{{\mathbb Z}_2} {\mathbb
S}^k$, and an equivariant poset map
\[\Cm \times_{{\mathbb Z}_2} P
\rightarrow_{{\mathbb Z}_2} \Hom(T,G),
\]
and hence by adjunction an equivariant poset map
\[
P \rightarrow_{{\mathbb Z}_2} \Poset_{{\mathbb Z}_2}\bigl(\Cm,
\Hom(T,G) \bigr) \simeq_{{\mathbb Z}_2} \Hom(T \times_{{\mathbb Z}_2}
\Cpm, G),
\]
where the last equivalence is an instance of \prettyref{thm:firstentry}.
It follows that
\[\coind_{{\mathbb Z}_2
}\real{\Hom(T \times_{{\mathbb Z}_2} \Cpm, G)}
  \geq \coind_{\Z_2}\real{P}=k,
\]
and hence the desired inequality.
\end{proof}

\begin{cor}\label{cor:mapsfrom}
If $\coind_{{\mathbb Z}_2}\bigl(\Hom(K_2, G)\bigr) \geq k$ then for sufficiently large $m$ there exists a graph homomorphism $T_{k,m} \rightarrow G$.
\qed
\end{cor}

\begin{rem}
We point out that the odd cycles $C_{2r+1}$ are
also Stiefel-Whitney test graphs that satisfy the equivariant coloring condition of \prettyref{lem:colorTkm}.  For the odd cycle $C_{2r+1}$ on vertices $\{1, \dots, 2r+1\}$, recall that the involution was given by $i \mapsto 2-i \pmod{2r+1}$ and hence in our coloring we map $r+1 \mapsto 1$ and $r+2 \mapsto 2$.  Hence one can take either $K_2$ or
$C_{2r+1}$ as the `base' graph and apply iterations of the $\qm
\times_{{\mathbb Z}_2} \Cpm$ construction to obtain new
Stiefel-Whitney test graphs (see \prettyref{fig:C5C3}).
\end{rem}

\begin{figure}[!ht]
\begin{center}
\input{C5C3.pspdftex}
\caption{The graphs $C_5, \Cptx6$, and $C_5 \times_{{\mathbb
Z}_2} \Cptx6$.}\label{fig:C5C3}

\end{center}
\end{figure}

\begin{rem}\label{rem:HomK2Tkm}
As was the case with the spherical graphs, one might ask if it would have been enough to know the
connectivity of $\Hom(K_2, T_{k,m})$ to establish $\chi(T_{k,m})\ge k+2$.  But here this is not so, since if $G$ is a graph with a right ${\mathbb Z}_2$-action
and $m\ge5$ (so that the left action on $\Cpm$ is $5$-discontinuous), \prettyref{thm:secondentry}  and \prettyref{ex:K2Cm} give us
\begin{align*}
\Hom(K_2, G \times_{\Z_2} \Cpm) &\simeq_{\Z_2} \Hom(K_2, G) \times_{\Z_2} \Hom(K_2, \Cpm)\\
&\simeq_{\Z_2} \Hom(K_2, G) \times_{\Z_2} \Hom({\bf 1}, \Cpm)\\
&\simeq_{\Z_2} \Hom(K_2, G) \times_{\Z_2} {\mathbb S}^1_b.
\end{align*}
\noindent
Hence $\Hom(K_2, T_{k,m}) \simeq_{\Z_2} \Hom(K_2,T_{k-1,m}) \times_{\Z_2} {\mathbb S}^1_b$ and by induction we get
\begin{align*}
\Hom(K_2, T_{k,m}) &\simeq \Hom(K_2,K_2) \times_{\Z_2} ({\mathbb S}^1_b \times_{{\mathbb Z}_2} \cdots \times_{{\mathbb Z}_2} {\mathbb S}^1_b)\\
&\approx_{{\mathbb Z}_2}
\underbrace{
{\mathbb S}^1_b \times_{{\mathbb Z}_2} \cdots \times_{{\mathbb Z}_2} {\mathbb S}^1_b
}_{k-\textrm{times}}
\end{align*}
for $k\ge1$ and $m\ge5$.

Hence these spaces have nontrivial fundamental group (are not 1-connected) for all~$k$.  In fact one can show that $\coind \bigl(\Hom(K_2, T_{k,m})\bigr) = 1$ for all~$k$
(the space
${\mathbb S}^1_b \times_{{\mathbb Z}_2} \cdots \times_{{\mathbb Z}_2} {\mathbb S}^1_b$
has the finite-dimensional contractible space $\R^k$ as a covering space and
hence its fundamental group does not contain an element of order~$2$),
so that these are examples of so-called \emph{non-tidy} ${\mathbb Z}_2$-spaces discussed in \cite{M03}.
\end{rem}

\subsection{Discussion}
We briefly discuss how one can view the spherical graphs and the twisted toroidal graphs as generalizations of odd cycles in the context of topological lower bounds on chromatic number.  In \cite{Schspace} the second author establishes a connection between the topology of $\Hom(K_2,\qm)$ and $\Hom (C_{2r+1},\qm)$ complexes.  On the one hand, as an easy consequence of his result we know that if $\Hom(K_2, G)$ admits an equivariant map from a circle (e.g., $\Hom(K_2, G)$ is connected), then there exists a homomorphism from an odd cycle $C_{2r+1} \rightarrow G$ for $r$ sufficiently large (so that $\Hom(C_{2r+1},G)$ is nonempty).  This in turn implies that the chromatic number of $G$ is at least 3.  Babson and Kozlov also showed that odd cycles are test graphs in the sense that higher connectivity of $\Hom(C_{2r+1},G)$ provides further bounds on $\chi(G)$.  These facts can be summarized as follows.
\begin{itemize}
\item
Connectivity of $\Hom(K_2, G)$ implies the existence of a homomorphism $C_{2r+1} \rightarrow G$.  More generally, if $\Hom(K_2, G)$ admits an ${\mathbb Z}_2$-equivariant map from a $k$-sphere, then for sufficiently large $r$, $\Hom(C_{2r+1}, G)$ admits an ${\mathbb Z}_2$-equivariant map from a $(k-1)$-sphere.

\item
Connectivity of $\Hom(C_{2r+1}, G)$ provides the `correct' lower bound on $\chi(G)$:
\[\conn \bigl(\Hom(C_{2r+1}, G)\bigr) \leq \Ht_{{\mathbb Z}_2} \bigl(\Hom(C_{2r+1}, G)\bigr) \leq \chi(G) -3 = \chi(G) - \chi(C_{2r+1}).\]
\end{itemize}

The results of our paper generalize this situation in the following way.  We construct a family of graphs $S_{k,m}$ (the family $T_{k,m}$ works equally well) with properties that resemble `higher dimensional' analogues of those of the odd cycle.  Namely, if $\Hom(K_2, G)$ admits a ${\mathbb Z}_2$-equivariant map from a $k$-sphere (e.g., is $(k-1)$-connected), then there exists a graph homomorphism $S_{k,m} \rightarrow G$ for sufficiently large $m$.  Furthermore, the $S_{k,m}$ have the additional property that they are test graphs in the sense that higher connectivity of $\Hom(S_{k,m}, G)$ provide the best possible bounds on $\chi(G)$.

\begin{itemize}
\item
$k$-connectivity of $\Hom(K_2, G)$ implies the existence of a homomorphism $S_{k,m} \rightarrow G$, and more generally if $\Hom(K_2, G)$ admits an ${\mathbb Z}_2$-equivariant map from a $j$-sphere, then for sufficiently large $m$, $\Hom(S_{k,m}, G)$ admits an ${\mathbb Z}_2$-equivariant map from a $(j-k)$-sphere.

\item
Connectivity of $\Hom(S_{k,m}, G)$ provides the `correct' lower bound on $\chi(G)$:
\[\conn \bigl(\Hom(S_{k,m}, G)\bigr) \leq \Ht_{{\mathbb Z}_2} \bigl(\Hom(S_{k,m}, G)\bigr) \leq \chi(G) - (k + 2) = \chi(G) - \chi(S_{k,m}).\]

\end{itemize}

In this sense a general pattern emerges: the $k$-connectivity of the original Lov\'{a}sz $\Hom(K_2,G)$ complexes implies the existence of homomorphisms from certain `$k$-dimensional' graphs into $G$, which in turn imply stronger lower bounds on the chromatic number of $G$ in terms of their connectivity.  In addition, as with odd cycles, the family of graphs $S_{k,r}$ are Stiefel-Whitney test graphs which `sit below' the original bound obtained by $K_2$ in the sense that
\[\Ht_{{\mathbb Z}_2}\bigl(\Hom(S_{k,r}, G)\bigr) + k \leq \Ht_{{\mathbb Z}_2}\bigl(\Hom(K_2, G)\bigr) \leq \chi(G) - 2.\]

Hence the topological bounds obtained from height of the $\Hom(S_{k,r},\qm)$ (or similarly the $\Hom(T_{k,m},\qm)$) complexes are no better than the original bounds obtained by Lov\'{a}sz.

\subsection{Generalized Mycielski Graphs}\label{sec:mycielski}
As a further application of our general methods, we show that test graphs can be obtained via the
generalized Mycielski-construction introduced in
\cite{orig-gen-myc}.
\begin{defn}[Generalized Mycielski construction]
\label{def:myc}
For $m\ge0$
let $P_m'$ be a path of length $m$ with a loop added at one end,
\begin{align*}
V(P_m')&=\{0,1,\dots,m\},\\
E(P_m')&=\{(n,m): \abs{n-m}=1\} \union \{(0,0)\}.
\intertext{For $m\ge1$ and $G$ a graph we define the graph}
M_m G&\deq (P_m'\times G) / (\set m\times V(G)).
\end{align*}
\end{defn}

\begin{prop}\label{prop:chi-myc-ub}
Let $G$ be a graph and $m\ge1$. Then
$\chi(M_m G)\le\chi(G)+1$.
\end{prop}

\begin{proof}
The graph obtained from $M_m G$ by removing the vertex
$\set m\times V(G)$ is isomorphic to $P'_{m-1}\times G$ and hence admits
a graph homomorphism to~$G$.
\end{proof}

In \cite{csorba-myc} Csorba showed that
$\Hom(K_2, M_m G)$ is homotopy equivalent to $\Sigma \Hom(K_2, G)$, the suspension of
$\Hom(K_2, G)$ (in fact Csorba worked with the neighborhood complex of $G$, but this is known to be homotopy equivalent to $\Hom(K_2,G)$).  It follows that $\Hom(K_2, M_m^k
K_2)\homot\susp^k\Hom(K_2, K_2)\homeo\susp^k\Sphere^0\homeo\Sphere^k$
and therefore $\chi(M_m^k K_2)=k+2$.  Hence we have examples of graphs with an arbitrarily large gap between the chromatic number and the odd girth, since the odd girth of $M_m^k K_2$ equals $2m+1$, independent
of~$k$.

In \cite{gen-myc} the authors mention the desirability of a $\Z_2$-version of this
result, and in \cite{gg-ky-fan} it was shown that
$\Hom(K_2, M_m G)\homot_{\Z_2}\susp\Hom(K_2, G)$.  To show that the
construction in fact yields test graphs, we will also have to understand the actions on $\Hom(K_2, M_m G)$ induced by actions
on~$G$.  On the other hand, to obtain our result we only need an equivariant map to
$\Hom(K_2, M_m G)$, and do not need to establish that it is a homotopy
equivalence.  For our purposes, the following is sufficient.

\begin{prop}\label{prop:susp-myc}
Let $G$ be a graph with at least one edge
and  $m\ge1$. There is an equivariant map
$\Sigma \real{\Hom(K_2, G)}\to_{\Z_2}\real{\Hom(K_2, M_m G)}$
which is natural with respect to automorphisms of~$G$.
\end{prop}

\begin{proof}
For a $\Z_2$-space $X$ we identify its suspension~$\susp X$ with the space
\[\left([-(m+1),m+1]\times X\right)
\big/
\left(\set{-(m+1)}\times X,\set{m+1}\times X\right),
\]
a $\Z_2$-space with action $\tau[(r,x)]=[(-r,\tau x)]$.
It is easily seen that $\real{\Hom(K_2, P'_m)}$ is homeomorphic to an interval.
Indeed,
there is a homeomorphism
\[h\colon[-m,m]\to_{\Z_2}\real{\Hom(K_2,P'_m)}\]
satisfying, for integers~$n$ with $-m\le n\le m$,
\[h(n)=
\begin{cases}
(0,0),&n=0,\\
(2s+1,2s),&n=2s+1>0,\\
(2s-1,2s),&n=2s>0,\\
(2s,2s+1),&n=-(2s+1)<0,\\
(2s,2s-1),&n=-2s<0,
\end{cases}
\]
where we write graph homomorphisms $K_2\to G$ as pairs of vertices of~$G$.
This yields a map
\begin{align*}
f\colon
[-m,m]\times\real{\Hom(K_2, G)}
&\to
\real{\Hom(K_2,P'_m)}\times\real{\Hom(K_2,G)}
\to\real{\Hom(K_2,P'_m\times G)}
\\&\to\real{\Hom(K_2,M_m G)}.
\end{align*}
The second map in this composition is induced by
$(\alpha,\beta)\mapsto(v\mapsto\alpha(v)\times\beta(v))$, compare
the beginning of the proof of
\prettyref{lem:bundle0'}.

We let $(A,B)$ denote an element of $\Hom(K_2,G)$, where $A,B \subset V(G)$.  The restriction $f\bigl((-1)^mm,\qm\bigr)$ maps $(A,B)\in\Hom(K_2,G)$ to
\[\bigl(\set{(m-1,a)\colon a\in A},\ast\bigr)
\in\Hom(K_2,M_m G),\]
where we identify a vertex $(n,v)\in V(P_m'\times G)$ with its unique image in
$M_m G$ if $n<m$ and denote the vertex of $M_m G$ which is the equivalence
class $\set m\times V(G)$ by~$\ast$.
Since each of these elements of $\Hom(K_2, M_m G)$ is less than or equal to
\[
\bigl(\set{(m-1,a)\colon \text{there is $u$ s.t.~$(a,u)\in E(G)$}}
,\ast\bigr)\in\Hom(K_2,M_m G),\]
the map $f\bigl((-1)^mm,\qm \bigr)$ extends from $\Hom(K_2, G)$ to a cone over
$\Hom(K_2, G)$.  Consequently, the map~$f$ extends to a $\Z_2$-map
$\susp \real{\Hom(K_2,G)}\to\real{\Hom(K_2,M_mG)}$.  This construction is
natural with respect to automorphisms of~$G$.
\end{proof}

\begin{cor}
Let $G$ be a graph with at least one edge.
Then \[\coind \Hom(K_2, M_m G)\ge\coind\Hom(K_2, G)+1
\]\
for all  $m\ge1$. \qed
\end{cor}

\begin{cor}\label{cor:Skb-to-myc}
Let $m\ge1$ and $k\ge0$.
There is a $(\Z_2\times\Z_2)$-map
$\Sphere^k_b\to\real{\Hom(K_2, M_m^k K_2)}$.
\end{cor}

\begin{proof}
We note that $\Sphere^{k+1}_b\homeo_{\Z_2\times\Z_2} \susp \Sphere^k_b$,
where the left action on $\Sphere^k_b$ is extended to $\susp \Sphere^k_b$ by
also changing the sign of the suspension parameter as in
\prettyref{prop:susp-myc}, while the right action is extended by
keeping the suspension parameter fixed, i.e.\ using the functoriality
of the suspension.  The result now follows from
\prettyref{prop:susp-myc} by induction, since
$\Hom(K_2,K_2)\homeo_{\Z_2\times\Z_2}\Sphere^0_b$.
\end{proof}
This corollary is sufficient to apply \cite[Thm.~4.6]{Schspace}
together with \prettyref{prop:chi-myc-ub} to obtain that the graphs
$M^k_m K_2$ are test graphs.  However, we repeat the argument here for the reader's convenience.
\begin{thm}\label{thm:myc-tg}
For $m\ge 1$ and $k\ge 0$, the graph $M^k_m K_2$ is a
Stiefel-Whitney test graph with $\chi(M^k_m K_2)=k+2$
\end{thm}

\begin{proof}
The inequality $\chi(M^k_m K_2)\le k+2$ follows by induction from
\prettyref{prop:chi-myc-ub}.  For any graph~$G$, we have maps
\begin{align*}
\Sphere^k_b\times_{\Z_2}\real{\Hom(M^k_m K_2, G)}
&\to_{\Z_2}\real{\Hom(K_2,M^k_m K_2)}\times_{\Z_2}\real{\Hom(M^k_m K_2, G)}
\\
&\to_{\Z_2}\real{\Hom(K_2, G)}.
\end{align*}
The first map we obtain from \prettyref{cor:Skb-to-myc}, the second by composition
of multihomomorphisms. So whenever $\Hom(M^k_m K_2, G)\ne\emptyset$, from \prettyref{prop:Sk} we have
\begin{align*}
\Ht_{\Z_2} \bigl(\Hom(M^k_m K_2, G)\bigr) +k+2
&\le \Ht_{\Z_2}\bigl(\Sphere^k_b\times_{\Z_2}\real{\Hom(M^k_m K_2, G)}\bigr)+2 \\
&\le \Ht_{\Z_2} \bigl(\Hom(K_2, G)\bigr) +2
\le\chi(G).
\end{align*}
Setting $G=K_{\chi(M^k_m K_2)}$, we obtain
$k+2\le\chi(M^k_m K_2)$.
\end{proof}

\section{Further applications}\label{sec:furtherapps}
In this section we discuss other applications of the structural results and constructions from above.  The spherical graphs lead us to a notion of generalized \emph{homomorphism duality}, while the twisted toroidal graphs have an application to a conjecture of Lov\'{a}sz.  We use \prettyref{thm:firstentry} to give graph-theoretical interpretations of
some topological invariants (${\mathbb Z}_2$-index and coindex) of
$\Hom$ complexes.  Finally, we use \prettyref{thm:secondentry} to obtain
the result regarding $S_n$-universality for $\Hom$ complexes mentioned
above.

\subsection{Homomorphism duality}\label{sec:duality}

Duality in homomorphisms of graphs and other relational structures has been extensively studied in the work of Ne\v{s}et\v{r}il and his coauthors (see for instance \cite{NT00}).  The basic idea is to identify a family ${\mathcal F}$ of obstructions to the existence of a homomorphism into a given graph $G$.  For us the exemplary example is the collection of all odd cycles ${\mathcal C}_{odd}$, a family which provides obstructions to homomorphisms to the edge $K_2$ (2-colorings).  For a collection ${\mathcal C}$ of graphs we let $\hom({\mathcal C},H)$ denote the set of graph homomorphisms $\{f:G \rightarrow H, G \in {\mathcal C}\}$.  This duality can then be expressed as:
\[\hom(\mathcal{C}_{odd}, G) \neq \emptyset \Leftrightarrow \hom(G, K_2) = \emptyset.\]

The primary focus in \cite{NT00} is the study of \textit{finite dualities} (where the family ${\mathcal F}$ is required to be a finite set).  Although an infinite set, ${\mathcal C}_{odd}$ represents a particularly nice family in that there exist homomorphisms $C_{2r+3} \rightarrow C_{2r+1}$ which form a \textit{linear} direct system:
\[\cdots \rightarrow C_{2r+3} \rightarrow C_{2r+1} \rightarrow \cdots \rightarrow C_5 \rightarrow C_3.\]

A naive hope would be to search for a similar direct system ${\mathcal F}_k$ (for $k > 1$) which provided obstructions to homomorphisms into larger complete graphs:
\[\hom({\mathcal F}_k, G) \neq \emptyset \Leftrightarrow \hom(G, K_{k+1}) = \emptyset.\]
\noindent Although we do not see any obvious reason why such a family could not exist, it does seem like a lot to hope for as it would for instance imply the long-standing conjecture of Hedetniemi which states that the chromatic number of the (categorical) product $G \times H$ is equal to the minimum of the chromatic numbers of $G$ and $H$.

Here we modify the `classical' duality picture and consider the \emph{space} $\Hom(K_2, G)$ of homomorphisms.  This leads us to a connection between homomorphism duality and the (equivariant) topology of $\Hom$ complexes.  We have seen that the topological complexity of $\Hom(K_2,G)$ provides a lower bound on chromatic number (and hence an obstruction to homomorphisms to complete graphs).  We then search for a family ${\mathcal S}_k$ of graphs which have the property that
\[\hom({\mathcal S}_k, G) \neq \emptyset \Leftrightarrow \Hom(K_2,G) \textrm{ has `complexity $k$'}.\]

Our construction of the spherical graphs $S_{k,m}$ in \prettyref{sec:spherical} (see \prettyref{def:Snk}) provides us with a candidate for such a family.  Recall that these graphs were defined as $S_{k,m} \deq K_2 \times{{\mathbb Z}_2} \bigl(F(X^k_m)\bigr)^1$, where $X^k_m$ denotes the $m$th barycentric subdivision of the boundary of the regular $k+1$-dimensional cross polytope, and $F(X^k_m)$ its face poset.  We have a map of posets $F(X^k_{m+1}) \rightarrow F(X^k_{m})$ which, for each $k$, gives a direct system of graphs
\[\dots \rightarrow S_{k,m+1} \rightarrow S_{k,m} \rightarrow \dots \rightarrow S_{k,0}.\]
\noindent
We use ${\mathcal S}_k \deq \bigl(S_{k,m}\bigr)_{m \geq 0}$ to denote this direct system of graphs.  Our result regarding generalized duality is then the following.

\begin{prop} \label{first}
For any graph $G$ we have
\[\hom({\mathcal S}_k, G) \neq \emptyset \Leftrightarrow \coind_{{\mathbb Z}_2} \bigl(\Hom(K_2,G)\bigr) \geq k.\]
\end{prop}

\begin{proof}
We apply \prettyref{prop:coind} with $T = K_2$.  For some $k \geq 0$, we have $\hom({\mathcal S}_k, G) \neq \emptyset$ if and only if there exists a graph homomorphism $S_{k,m} \rightarrow G$ for some $m$.  By \prettyref{prop:coind} this is the case exactly when we have $\coind_{{\mathbb Z}_2} \bigl(\Hom(K_2,G)\bigr) \geq k$.  The result follows.
\end{proof}

\subsection{A conjecture of Lov\'{a}sz}\label{sec:Lovconj}
We next see how the twisted toroidal graphs from \prettyref{sec:toroidal} (see \prettyref{defn:Tgraphs}) lead us to a proof of a weakened version of the following conjecture, attributed to Lov\'{a}sz.

\begin{conj}\label{conj:Lovconjecture}
Let $G$ be a graph with no loops and at least one edge.  If $\Hom(T,G)$ is empty or
$k$-connected for all graphs $T$ of maximum degree $\leq d$, then
$\chi(G) \geq k + d + 2$.
\end{conj}

In \cite{BW04} Brightwell and Winkler have managed to prove a weaker
version of this conjecture for the case $k = 0$.
\begin{thm}[Brightwell and Winkler]
If $\Hom(T,G)$ is empty or connected for all graphs $T$ of maximum
degree $\leq d$ then $\chi(G) \geq \frac{d}{2} + 2$.
\end{thm}

We note that the $d=2$ case of \prettyref{conj:Lovconjecture}
follows from what we already know.  If $\Hom(K_2,G)$ is
$k$-connected, then $\Hom(C_{2r+1}, G)$ is nonempty for some $r$.
Hence the space $\Hom(C_{2r+1},G)$ is $k$-connected by assumption, and so
$\chi(G) \geq k + 4 = k + d + 2$ by the Babson-Kozlov result. We can
apply this same simple argument with the graphs $T_{k,m}$ to get the
following (weakened) version of the original conjecture.

\begin{prop}
Suppose $G$ is a graph with no loops and at least one edge.  If $\Hom(T,G)$ is empty or $k$-connected for every graph~$T$ with
maximum degree $\leq d$, then
\[\chi(G) \geq \min\{k+1, \lfloor \log_3 d \rfloor\} +
k + 3.\] \noindent In particular, if $d=3^{k+1}$, we have $\chi(G)
\geq 2k + 4$.
\end{prop}

\begin{proof}
Since $G$ has at least one edge, we have that $\Hom(K_2,G)$ is
nonempty and hence by assumption is $k$-connected.  So then we have
$\coind_{{\mathbb Z}_2} \Hom(K_2, G) \geq k+1$, and hence by \prettyref{cor:mapsfrom}
we have a graph homomorphism $T_{j,m} \rightarrow G$ for some
$m$, for all $j \leq k+1$.  From \prettyref{prop:Tkmgraphs} we have that $T_{j,m}$ has maximum degree
$d_j = 3^j$. We take $j = \min \{k+1, \lfloor \log_3 d \rfloor\}$, and by assumption we get that $\Hom(T_{j,m},G)$ is
$k$-connected.  Since the $T_{j,m}$ are Stiefel-Whitney test graphs
(and hence homotopy test graphs) this implies that
\[\chi(G) \geq \chi(T_{j,m}) + k + 1 = j + 2 + k + 1 = \min\{k+1, \lfloor \log_3 d \rfloor\} + k + 3.\]

\noindent In the case $d = 3^{k+1}$ we take $j = k+1$ and get
$\chi(G) \geq 2k + 4$.
\end{proof}

\begin{rem}
We note that in the case of $T_{1,m}$  we have subgraphs of smaller degree (namely the odd cycles) which suffice to provide the desired topological bounds on chromatic number.  It would be interesting to find similar subgraphs of $T_{k,m}$ for larger $k$.
\end{rem}

\subsection{Index and coindex of $\Hom$ complexes}
In \prettyref{prop:coind} we saw that homomorphisms from the spherical graphs gave a graph-theoretic interpretation of the coindex of $\Hom$ complexes.  In fact we can use \prettyref{thm:firstentry} to give a graph-theoretical interpretation of the \emph{space} of all ${\mathbb Z}_2$-maps from a $k$-sphere (with the antipodal action) into the complex $\Hom(K_2,G)$.  Recall that $S_{k,m}$ are the spherical graphs defined in \prettyref{def:Snk}.

\begin{prop}
Let $T$ be a graph with a right ${\mathbb Z}_2$-action.  Then we
have
\[\colim_m \Hom(S_{k,m}, G)
\simeq_{{\mathbb Z}_2} \Map_{{\mathbb Z}_2}\bigl({\mathbb S}^k_b,
\Hom(K_2,G)\bigr),\] where the direct system that defines the colimit is
described in the proof.
\end{prop}

\begin{proof}
The proof runs along the lines of \prettyref{prop:coind}, and so we provide only a sketch here.  A similar statement for the case $k=1$ was proved in \cite{Schspace} and a non-equivariant version was established in \cite{DocGro}.

We apply \prettyref{thm:firstentry}, with the spherical graphs $S_{k,m}\deq K_2\times_{\Z_2}\bigl(F(X^k_m)\bigr)^1$ in place of $T \times_{\Z_2} P^1$.  To
complete the proof we need to describe the direct system involved
in the colimit.  For a fixed $k$ we have a natural poset map $F(X^{k}_{m+1}) \rightarrow F(X^k_m)$ which induces a graph homomorphism $S_{k,m+1} \rightarrow S_{k,m}$.  This homomorphism respects each of the ${\mathbb Z}_2$-actions, and gives us our desired direct system.  Since any ${\Z_2}$-map from a $k$-sphere into $\Hom(T,G)$ can be approximated up to homotopy as a map from $X^k_m$ for $m$ sufficiently large, the result follows.
\end{proof}

In addition, we can use the observation from \prettyref{rem:graphversion} along with the fact that $\Hom(K_2, K_n)$ is a sphere (see \cite{BKcom}) to give the following interpretation of the \emph{index} of the free ${\mathbb Z}_2$-space $\Hom(T, G)$.  For a graph $G$ let $BG$ be the graph whose vertices are cliques of looped vertices of $G$ and in which two cliques are neighbors if one of them contains the other; $B^{i+1}G$ can be thought of as the looped 1-skeleton of the $i$th barycentric subdivision of the order complex of $\Hom({\bf 1}, G)$.  We then get the following.

\begin{prop}\label{prop:ind}
Let $T$ be a graph with an involution that flips an edge.  Then the index of $\Hom(T, G)$ is the minimum of the chromatic numbers of $B^i(G^T) \times_{\Z_2} K_2 $, $i\ge0$.
\end{prop}

\begin{proof}
We make the appropriate substitutions in \prettyref{rem:graphversion} to obtain
\begin{equation}\label{eq:index}
\Poset_{{\mathbb Z_2}} \bigl(\Hom({\bf 1},G^T),\Hom(K_2,K_n)\bigr) \simeq \Hom \bigl(S(G^T) \times_{{\mathbb Z}_2} K_2, K_n\bigr).
\end{equation}
Hence, given a ${\mathbb Z}_2$-equivariant map
\[|\Hom({\bf 1}, G^T)| \simeq_{{\mathbb Z}_2} |\Hom(T,G)| \rightarrow_{{\mathbb Z}_2} {\mathbb S}^n \simeq |\Hom(K_2,K_n)|,\]
\noindent
we use simplicial approximation to obtain an equivariant map
\[\Chain^i \bigl(\Hom({\bf 1}, G^T) \bigr) \rightarrow_{{\mathbb Z}_2} \Hom(K_2, K_n).\]
\noindent
But we have $\Hom({\bf 1}, B^i(G^T)) = \Chain^i \big(\Hom({\bf 1}, G^T)\big)$, so that
\[\Hom({\bf 1}, B^i(G^T)) \rightarrow_{{\mathbb Z}_2} \Hom(K_2,K_n).\]
\noindent
Applying \prettyref{eq:index} provides us with a coloring $B^i(G^T) \rightarrow K_n$.

On the other hand if we have a coloring $B^i(G^T) \times_{{\mathbb Z}_2} K_2 \rightarrow K_n$, then \prettyref{eq:index} gives us an equivariant map
\[\Chain^i \bigl(\Hom({\bf 1}, G^T)\bigr) = \Hom({\bf 1}, B^i(G^T)) \rightarrow_{{\mathbb Z}_2} \Hom(K_2,K_n) \simeq_{{\mathbb Z}_2} {\mathbb S}^n,\]
\noindent
and hence an equivariant map $|\Hom(T,G)| \rightarrow_{{\mathbb Z}_2} {\mathbb S}^n$.
\end{proof}

\section{Proofs}\label{sec:proofs}
In this section we provide the complete proofs of \prettyref{thm:firstentry} and \prettyref{thm:secondentry}, out main structural results mentioned in the first section of the paper.

\subsection{Proofs of \prettyref{thm:firstentry} and related lemmas}
We begin with \prettyref{thm:firstentry}.  For this we first establish certain auxiliary results which are used in the proof of the theorem at the end of the section.  We need the following construction.
\begin{defn}\label{defn:Homaction}
Let $\Gamma$ be a group and suppose $G$ and $H$ are graphs with left $\Gamma$-actions.  The poset $\Hom(G,H)$ then has a left $\Gamma$-action given by
\[(\gamma \cdot \alpha)(g) = \gamma\cdot \bigl(\alpha(\gamma^{-1}\cdot g)\bigr),\]
for $\alpha \in \Hom(G,H)$ and $\gamma \in \Gamma$.  Define $\Hom_{\Gamma}(G,H)$ to be the subposet of $\Hom(G,H)$ given by the fixed point set of this action.
\end{defn}
We note that the equivariant graph homomorphisms $f:G \rightarrow_{\Gamma} H$ are (in general only some of the) atoms of $\Hom_{\Gamma}(G,H)$.

\begin{lemma} \label{lemma:EquiAdjoint}
Let $T$, $G$, and $H$ be graphs and suppose $\Gamma$ is a group acting from the right on $T$, on the left on $H$, and hence on the right on $T \times H$ by $(t,h) \cdot \gamma = (t \gamma^{-1}, \gamma h)$.  Then the poset $\Hom(T \times_{\Gamma} H, G)$ can be included in $\Hom_{\Gamma}(H, G^T)$ so that $\Hom(T \times_{\Gamma} H, G)$ is the image of a closure map on $\Hom_{\Gamma}(H, G^T)$.  In particular, we have the inclusion of a strong deformation retract of posets
\begin{center}
$\xymatrix{ \Hom(T \times_{\Gamma} H, G) \ar@{^{(}->}[r]_-\simeq & \Hom_{\Gamma}(H, G^T). }$
\end{center}
\end{lemma}

\begin{proof}
We recall the proof of $\Hom(T \times H, G) \rightarrow \Hom(H,G^T)$ from \cite{DocHom}.  Define a map of posets $\varphi: \Hom(T \times H, G) \rightarrow \Hom(H,G^T)$ by
\[\varphi(\alpha)(h) = \{f:V(T)\rightarrow V(G) : f(t) \in \alpha(t,h), \hspace{.1 in} \forall t \in T \},\]
\noindent for $\alpha \in \Hom(T \times H, G)$ and $h \in H$.

We next construct a map of posets $\psi: \Hom(H,G^T) \rightarrow \Hom(T \times H, G)$ according to
\[\psi(\beta)(t,h) = \{f(t): f \in \beta(h) \},\]
\noindent
for $\beta \in \Hom(H,G^T)$, and $(t,h) \in T \times H$.  In \cite{DocHom} it is shown that $\phi$ and $\psi$ are both well-defined and that $\psi \circ \varphi = \id$ and $\varphi \circ \psi \geq \id$.  Hence the map $\psi$ induces a homotopy equivalence.

Next we note that $\Gamma$ acts on $G^T$ and hence we have an action
on the poset $\Hom(H,G^T)$ as described in \prettyref{defn:Homaction},
where $\Hom_{\Gamma}(H,G^T)$ is the fixed point poset.  Also, $\Gamma$
acts on $T \times H$ and hence on $\Hom(T \times H, G)$
according to $(\gamma \cdot \alpha)(t,h) = \alpha ((t,h)\cdot
\gamma)$.  We let $\Hom_{\Gamma}(T \times H, G)$ denote the subposet
consisting of the fixed points of this action.  The construction of
$\varphi$ and $\psi$ are both natural with respect to the
automorphisms of $T$ and $H$, and hence are equivariant with respect
to the $\Gamma$ actions on the posets.  Hence the restriction of
$\varphi$ induces a homotopy equivalence on the fixed point sets
$\Hom_{\Gamma}(T \times H, G) \rightarrow \Hom_{\Gamma}(H,G^T)$.

Finally we claim that $\Hom_{\Gamma}(T \times H,G) \cong \Hom(T
\times_{\Gamma} H, G)$ (an isomorphism of posets), which would
complete the proof.  For this note that the quotient map
$p\colon T\times H\to (T\times H)/\Gamma=T\times_\Gamma H$
induces an injective poset map
$p^G:\Hom(T\times_{\Gamma} H, G) \rightarrow \Hom(T \times H, G)$.
The image
$\im(p^G)$ is contained in the subposet $\Hom_{\Gamma}(T \times H,G)$
since if $\alpha \in \Hom(T \times_{\Gamma} H, G)$ we have
$\bigl(\gamma \cdot (p^G(\alpha))\bigr)(t,h) = p^G(\alpha)\bigl((t,h)
\cdot \gamma)\bigr) = p^G(\alpha)(t \cdot \gamma^{-1}, \gamma \cdot h)
= p^G(\alpha)(t,h)$.  To show that $p^G$ is surjective onto
$\Hom_{\Gamma}(T \times H,G)$ suppose $\alpha \in \Hom_{\Gamma}(T
\times H, G)$ and let $\tilde \alpha([(t,h)]) \deq \alpha(t,h)$.
This is well-defined, since if $[(t,h)]=[(t',h')]$ then there is a
$\gamma\in\Gamma$ such that $(t,h)\cdot\gamma=(t',h')$
and therefore $\alpha(t',h')$ = $(\gamma\cdot\alpha)(t,h)=\alpha(t,h)$,
the latter equality since $\alpha$
is in the fixed point set.
Also $\tilde\alpha$ is an element of $\Hom(T \times_{\Gamma} H, G)$,
since if $[(t,h)]
\sim [(t^\prime, h^\prime)]$ we have by definition of a quotient graph
$(t,h)\cdot \gamma \sim
(t^\prime, h^\prime) \cdot \gamma^\prime$ for some $\gamma,
\gamma^\prime \in \Gamma$.
It follows that
$\tilde\alpha([t,h])\times\tilde\alpha([t',h'])
=\alpha((t,h)\cdot\gamma)\times\alpha((t',h')\cdot\gamma')
\subset E(G)$, and
$\tilde \alpha$ is a multihomomorphism with $p^G(\tilde \alpha) = \alpha$.
\end{proof}

\begin{prop} \label{prop:EquiPosetHom}
Suppose $G$ is a graph with a left $\Gamma$-action and $P$ is a poset with a left $\Gamma$-action, for some group $\Gamma$.  Then there is a closure map on the level of posets that induces a homotopy equivalence
\begin{center}
$\Hom_{\Gamma}\bigl(P^1, G\bigr) \simeq \Poset_{\Gamma}\bigl(P, \Hom({\bf 1},G)\bigr)$.
\end{center}
\end{prop}

\begin{proof}
We first note that $\Gamma$ acts on $\Poset\bigl(P, \Hom({\bf 1}, G)\big )$ according to $(\gamma \cdot f)(x)({\bf 1}) = \big \{ \gamma \cdot y : y \in (f(\gamma^{-1} \cdot x))({\bf 1}) \big \}$.  The poset $\Poset_{\Gamma}\bigl(P, \Hom({\bf 1}, G) \bigr)$ is the subposet corresponding to the fixed point set of this action.  As in the proof of \prettyref{lemma:EquiAdjoint}, $\Gamma$ acts on $\Hom\bigl(P^1, G \bigr)$ according to $(\gamma \cdot \alpha)(x) = \big \{\gamma \cdot y: y \in \alpha(\gamma^{-1} \cdot x) \big \}$ for $\gamma \in \Gamma$, $\alpha \in \Hom \bigl(P^1, G \bigr)$, and $x \in P^1$.   By definition (see \prettyref{defn:Homaction}) $\Hom_{\Gamma}\bigl(P^1, G \bigr)$ is the fixed point set.

We define a map of posets $\varphi:\Hom \bigl(P^1, G \bigr)
\rightarrow \Poset \bigl(P, \Hom({\bf 1}, G)\bigr)$, according to
\[\varphi(\alpha)(x)({\bf 1}) = \bigcup_{y \leq x, \textrm{ $y$ an atom}} \alpha(y),\]
\noindent for $x \in P$, $\alpha \in \Hom \bigl(P^1,G\bigr)$, and $x
\in P$.

To check that $\varphi(\alpha)(x) \in \Hom({\bf 1}, G)$ let $z \in \alpha(y)$ and $z^\prime \in \alpha(y^\prime)$
where $y, y^\prime \leq x $ are atoms.  Then $y \sim y^\prime$ in
$P^1$ and hence $z \sim z^\prime$ in $G$ as desired.  We see that $\varphi(\alpha)$ is a poset map since if $x \leq x^\prime$ in $P$,
then $y \leq x^\prime$ for every atom $y$ with $y \leq x$, and hence
$\varphi(\alpha)(x) \leq\varphi(\alpha)(x^\prime)$.  To see that $\varphi$ itself is a poset map we suppose
$\alpha \leq \alpha^\prime$ in $\Hom\bigl(P^1, G\bigr)$.  Since $\alpha \leq \alpha^\prime$, we have $\alpha(y) \subset
\alpha^\prime(y)$ for any $y \in P^1$.  Hence we have
\[\varphi(\alpha)(x) = \bigcup_{y \leq x} \alpha(y) \subset
\bigcup_{y \leq x} \alpha^\prime(y) = \varphi(\alpha^\prime)(x),\]
\noindent
and we conclude $\varphi(\alpha) \leq \varphi(\alpha^\prime)$.

Next we define a map $\psi: \Poset \bigl(P, \Hom({\bf 1}, G)\bigr) \rightarrow \Hom(P^1, G)$ according to $\psi(f)(x) = f(x)({\bf 1})$, for $f \in \Poset \bigl(P, \Hom({\bf 1}, G \bigr)$ and $x \in P^1$.  It is clear that $\psi \circ \varphi = \id$, and also $\varphi \circ \psi \leq \id$ since
\[\varphi \circ \psi (f)(x)({\bf 1}) = \bigcup_{y \leq x, \textrm{ $y$ an atom}} \psi(f)(y) = \bigcup_{y \leq x, \textrm{ $y$ an atom}} f(y)({\bf 1}),\]
\noindent
and $f(y)({\bf 1}) \subset f(x)({\bf 1})$ for all $y \leq x$.

The maps $\varphi$ and $\psi$ are natural with respect to automorphisms of $P$ and $G$, and hence are equivariant with respect to the $\Gamma$ actions described above.  The  restriction of $\varphi$ provides the desired homotopy equivalence.
\end{proof}

\begin{lemma} \label{lemma:closureclosure}
Suppose $Q$ is a poset, and $c:Q \rightarrow Q$ is a closure map
(say with $c(q) \geq q$ for all $q \in Q$).  Then given a poset $P$,
the induced maps $c_*: \Poset(P,Q) \rightarrow \Poset(P,Q)$ and
$c^*:\Poset(Q,P) \rightarrow \Poset(Q,P)$ are both closure maps.
\end{lemma}

\begin{proof}
For the first case, suppose $\phi \in \Poset(P,Q)$.  Then we have
$c_*(\phi)(p) = c(\phi(p)) \geq \phi(p)$ since $c$ is a closure map.
Hence $c_*(\phi) \geq \phi$.

For the other case, suppose $\psi \in \Poset(Q,P)$.  Then we have
$c^*(\psi)(q) = \psi(c(q)) \geq \psi(q)$ since $\psi$ is a poset map
and $c(q) \geq q$.  Hence $c^*(\psi) \geq \psi$.
\end{proof}

We can now string together these results to provide the proof of \prettyref{thm:firstentry}.

\begin{proof}[Proof of \prettyref{thm:firstentry}]
For $\Gamma, P$, and $T$ and $G$ as above we have homotopy equivalences:
\begin{align*}
\Hom(T \times_{\Gamma} P^1, G) &\simeq \Hom_{\Gamma}(P^1, G^T) \hspace{.2 in}
&&(\textrm{by \prettyref{lemma:EquiAdjoint}})\\
&\simeq \Poset_{\Gamma}(P, \Hom({\bf 1}, G^T)) \hspace{.2 in}
&&(\textrm{by \prettyref{prop:EquiPosetHom}} )\\
&\simeq \Poset_{\Gamma}(P, \Hom(T,G)) \hspace{.2 in}
&&(\textrm{by \prettyref{lemma:closureclosure}}).
\end{align*}
\end{proof}

\prettyref{thm:firstentry} shows that the functor $T \times_{\Gamma} (\qm)^1$ from the category of posets to graphs provides an `approximate' left adjoint to the functor $\Hom(T,\qm)$.  See the last section of the paper for further discussion regarding the categorical context of this statement.

\begin{rem}\label{rem:graphversion}
If we have $P \deq \Hom({\bf 1}, B^A)$ for graphs $A$ and $B$, we see that $P$ is the face poset of the clique complex (on the looped vertices) of the graph~$B^A$.  In this case $P^1 = S(B^A)$, where $S$ is the functor that takes the induced subgraph on looped vertices.  From \cite{DocHom} we have $\Hom({\bf 1}, B^A) \simeq \Hom(A,B)$ and we obtain:
\begin{align*}
\Poset_{\Gamma}\bigl(\Hom(A,B),\Hom(T,G)\bigr) &\simeq
\Poset_{\Gamma}\bigl(\Hom({\bf 1}, B^A), \Hom(T,G)\bigr)
&&\text{ (by \prettyref{lemma:closureclosure})} \\
&\simeq \Hom \bigl((\Hom({\bf 1}, B^A))^1 \times_{\Gamma} T,
G\bigr)
&&\text{(by \prettyref{thm:firstentry})}\\
&= \Hom\bigl(S(B^A) \times_{\Gamma} T, G\bigr).
\end{align*}
\end{rem}

\subsection{Proofs of \prettyref{thm:secondentry} and related lemmas}
We next turn to the proof of \prettyref{thm:secondentry}.  Once again this amounts to establishing some auxiliary results which we combine in the proof at the end of the section.  We begin by establishing conditions on a group action on a poset which guarantees that
realization commutes with taking quotients.  The following definition
is taken from \cite{BKgroup}.
\begin{defn}\label{def:regular}
Suppose a group~$\Gamma$ acts freely on a poset~$P$.  We call the
action \emph{strongly regular} if for all $u,v,m\in P$ and
$\gamma \in\Gamma$ such that $u,v\le m$ and $\gamma \cdot u=v$ we have $\gamma =e$.
\end{defn}
The definition given in \cite{BKgroup} (property SRP) also applies to
non-free actions, although this requires a slightly different formulation.
In addition, the authors of \cite{BKgroup} consider the more general situation of group actions on small categories.  We are only interested in the special case of free actions on posets, and for the reader's convenience provide a proof of the following lemma.
\begin{lem}\label{lem:regular}
Let $P$ be a finite poset on which the group~$\Gamma$ acts freely and
strongly regularly.  Then the map $P\to P/\Gamma$ induces a
homeomorphism $\real P/\Gamma\xto\homeo\real{P/\Gamma}$.
\end{lem}

\begin{proof}
We denote the quotient map by $\pi\colon P\to P/\Gamma$.  Let $p,q\in
P$, $p<q$.  We show that $\pi(p)<\pi(q)$.  By definition of $P/\Gamma$
we have $\pi(p)\le\pi(q)$.  Now if also $\pi(q)\le\pi(p)$ then there
is $\gamma \in\Gamma$ such that $\gamma \cdot q\le p\le q$.  But since $P$ is
finite and multiplication by~$\gamma$ an isomorphism $\gamma\cdot q\le q$
implies $\gamma\cdot q=q$ and hence $p=q$, contradicting $p<q$.  Therefore
$\pi$ maps each chain in $P$ injectively to $P/\Gamma$ and the
restriction of the map $\real\pi\colon\real P\to\real{P/\Gamma}$ to
each simplex is injective.  To show that the map $\real
P/\Gamma\to\real{P/\Gamma}$ is bijective and hence (by compactness) a
homeomorphism it therefore suffices to show that the preimage of each
$k$-simplex of $\real{P/\Gamma}$ consists of exactly one orbit of
$k$-simplices of $\real P$.  Let $\pi(p_0)<\pi(p_1)<\dots<\pi(p_k)$ be
such a simplex of~$P$.  Since the action of $\Gamma$ on $P$ is
strongly regular and free, there is a unique $\gamma_{k-1}\in\Gamma$ such
that $\gamma_{k-1}\cdot p_{k-1}<p_k$.  Repeating this argument, we obtain unique $\gamma_0,\dots,\gamma_{k-1}\in\Gamma$ such that $\gamma_0\cdot p_0<\gamma_1\cdot
p_1<\dots<\gamma_{k-1}\cdot p_{k-1}<p_k$.  The chains of~$P$ which are
mapped to $\set{\pi(p_0),\pi(p_1),\dots,\pi(p_k)}$ are therefore
exactly those of the form $\set{\eta \gamma_0\cdot p_0,\eta \gamma_1\cdot
p_1,\dots,\eta \gamma_{k-1}\cdot p_{k-1},\eta \cdot p_k}$ for $\eta\in\Gamma$.  This
concludes the proof.
\end{proof}

\begin{prop}
\label{prop:bundle0}
Let $T$, $G$ be graphs and $\Gamma$ a group acting on~$G$.
Let $p\colon G \rightarrow G/\Gamma$ denote the projection.  Assume
that $T$ is finite, connected, has at least one edge, and that a spanning tree~$S$ of~$T$ has been chosen.  Each edge of~$T$ which is not in~$S$ defines a cycle in~$T$ consisting of that edge and a simple path in~$S$.
Let $L$ be the set of the lengths of these cycles.
Assume that for all $v\in V(G)$, $\gamma \in\Gamma$ and $l\in L\unite\set{4}$
such that there is a (not necessarily simple)
path of length~$l$ in~$G$ from~$v$ to~$\gamma \cdot v$ we have $\gamma =e$.  Then the induced action of~$\Gamma$
on~$\Hom(T,G)$ is strongly regular and free, the map of posets $p_T:\Hom(T,G) \rightarrow \Hom(T,G/\Gamma)$ is rank preserving, and the induced map
\[\bar p_T\colon\Hom(T,G)/\Gamma\to\Hom(T,G/\Gamma)\]
is an isomorphism.
\end{prop}

\begin{rem}
We have $1\in L$ if and only if $T$ has a looped vertex.  We never
have $2\in L$.
\end{rem}

\begin{proof}
Let $\alpha\in\Hom(T, G)$, $u\in V(T)$, $v\in V(G)$, $\gamma\in\Gamma$, and
$v,\gamma v\in\alpha(u)$.  Since there is an edge incident with~$u$, it
follows that there is a path of length~$2$, and hence also one of
length~$4$, from $v$ to $\gamma v$.  Thus $\gamma =e$, which shows that the map
$\Hom(T,G)\to\Hom(T,G/\Gamma)$ is rank preserving and that the action
of $\Gamma$ on $\Hom(T, G)$ is strongly regular and free.

Now let $\beta\in\Hom(T, G/\Gamma)$, $u\in V(T)$, $v\in V(G)$, and
$p(v)\in\beta(u)$.
We will show that $\beta$ can be lifted uniquely starting in~$v$,
i.e.\ that there is
a unique $\alpha\in\Hom(T, G)$ with $\beta=p_T(\alpha)$ and
$v\in\alpha(u)$.  This will immediately imply that $\bar p_T$ is
a bijection.

We first show that there is a~$\beta_0\le\beta$ that lifts uniquely.
Choose any $\beta_0\le\beta$ of rank~$0$, i.e. let $\beta_0$ be a
graph homomorphism.  By the same argument as in the first paragraph, edges
lift uniquely with the lift of one of the vertices prescribed.  By
induction, the restriction of~$\beta_0$ to the tree~$S$ has a unique
lift~$\alpha_0$.  This is a lift of $\beta_0$, because if
$(w,w')\in E(T)\wo E(S)$, then there must be a $\gamma \in\Gamma$ such that
$(\alpha_0(w),\gamma \cdot\alpha_0(w'))\in E(G)$.  But now there is an $l\in
L$ and a path $w=w_1,w_2,\ldots,w_l=w'$ in~$T$, so
$\gamma \cdot\alpha_0(w_l),\alpha_0(w_1),\alpha_0(w_2),\ldots,\alpha_0(w_l)$
is a path of length~$l$ in~$G$ and we must have $\gamma =e$.  Next we show
that if some $\beta'<\beta$ has a unique lift, there is a $\beta''$
with $\beta'<\beta''\le\beta$ which also has unique lift.  Let
$\alpha'$ be the unique lift of $\beta'$, and $w\in V(T)$,
$z\in\beta(w)\wo\beta'(w)$.  Let $n\in V(T)$ be a neighbor of~$w$ and
$\tilde x\in\alpha'(n)$.  Since $(p(\tilde x),z)\in E(G/\Gamma)$,
there is a unique $\tilde z\in p^{-1}[\set z]$ with $(\tilde x, \tilde
z)\in E(G)$, and any lift $\alpha''$ of $\beta''$ which extends
$\alpha'$ must satisfy $\alpha''(w)=\alpha'(w)\unite\set{\tilde z}$
and agree with $\alpha'$ on the other vertices of~$T$.  We will show
that the function $\alpha''$ thus defined is in~$\Hom(T, G)$.  First
note that if $w$ is looped then so is~$z$ and hence also~$\tilde z$,
because otherwise there would be a path of length~$1$ between
different vertices of $p^{-1}[\set z]$.  Now let $n'$ be a neighbor
of~$w$, $\tilde x'\in\alpha'(n')$.  Also choose $\tilde
z'\in\alpha'(w)$.  Since $p(\tilde z)\in\beta'(w)$, $p(\tilde
x')\in\beta'(n')$, there is a $\gamma \in\Gamma$ such that $(\tilde
z,\gamma\cdot\tilde x')\in E(G)$.  Now $\tilde x',\tilde z',\tilde x,\tilde
z,\gamma\cdot\tilde x'$ is a path of length~$4$ and hence $\gamma =e$, which
means that $\tilde x'$ is a neighbor of~$\tilde z$ as required.

Finally, if $\alpha\in\Hom(T, G)$ and
$\beta'\le p_T(\alpha)$, then
$\alpha'(w)\deq\alpha(w)\intersect p^{-1}[\beta'(w)]$ defines a lift
of~$\beta'$ and $\alpha'\le\alpha$.  This shows that the poset
map~$\phi$ is not only a bijection, but in fact an isomorphism.
\end{proof}

\begin{cor}\label{cor:bundle}
Let $T$ be a finite tree without loops, and suppose $G$ is a graph with a $\Z_2$-action with the property that for all $v \in G$, there is no path of length four from $v$ to $\gamma v$, its image under the action.  Then the induced action of~$\Z_2$ on~$\Hom(T,G)$ is strongly regular and free, the map of posets $p_T\colon\Hom(T,G) \rightarrow \Hom(T,G/\Z_2)$ is rank preserving, and the induced map
\[\bar p_T\colon\Hom(T,G)/\Z_2\to\Hom(T,G/\Z_2)\]
is an isomorphism.
\qed
\end{cor}

\begin{lemma}
\label{lem:bundle0'}
Let $T$, $G$, and $H$ be graphs and suppose $\Gamma$ is a group acting from the right on~$G$ and from the left on~$H$,
and hence on $G\times H$ by $\gamma \cdot(u,v) \deq (u \gamma^{-1},\gamma v)$.  Then there is
a homotopy equivalence of posets
\[
\Hom(T, G\times H)/\Gamma\homot\Hom(T, G)\times_\Gamma\Hom(T, H)
\]
natural with respect to automorphisms of~$T$ and
$\Gamma$-automorphisms of $G$ and~$H$.  Also, if the $\Gamma$-action
on $\Hom(T, G\times H)$ is strongly regular and free, then so is the
$\Gamma$-action on~$\Hom(T, G)\times\Hom(T, H)$.
\end{lemma}

\begin{proof}
We recall the proof of $\Hom(T, G\times H)\homot\Hom(T,
G)\times\Hom(T, H)$ from \cite{DocHom}.  Let $p^1\colon G\times H\to G$, $p^2\colon
G\times H\to H$ denote the projections.  Let
\begin{align*}
\varphi\deq(p^1_T,p^2_T)\colon
\Hom(T, G\times H)&\to\Hom(T,G)\times\Hom(T, H)
\intertext{and}
\rho\colon\Hom(T,G)\times\Hom(T, H)&\to\Hom(T, G\times H)\\
(\alpha,\beta)&\mapsto (v\mapsto \alpha(v)\times \beta(v)).
\end{align*}
Then $\varphi\cmps\rho=\id$, $\rho\cmps\varphi\ge\id$, and since the
maps are natural with respect to automorphisms of the graphs,
$\varphi$~induces the desired homotopy equivalence.  Also, $\rho$ is
injective, and so the last statement follows.
\end{proof}

\begin{proof}[Proof of \prettyref{thm:secondentry}]
If, for some~$r$, the action on $P^1$ is $r$-discontinuous then so is
the diagonal action on $T\times P^1$.  Therefore by
\prettyref{prop:bundle0} we have \[\Hom(G, T \times_{\Gamma} P^1)\isom \Hom(G,
T \times P^1) / \Gamma\] and the action on $\Hom(G, T\times P^1)$ is
strongly regular and free.  By \prettyref{lem:bundle0'}, \[\Hom(G, T
\times P^1) / \Gamma \homot \Hom(G, T) \times_{\Gamma} \Hom(G, P^1),\]
with the quotient on the right hand side obtained from an action that is strongly regular
and free.  Consequently we get
\begin{align*}
\real{\Hom(G, T \times_{\Gamma} P^1)} &\homeo \real{\Hom(G, T \times P^1) / \Gamma}\\
&\homot \real{\Hom(G, T) \times_{\Gamma} \Hom(G, P^1)}\\
&\homeo \real{\Hom(G, T)} \times_{\Gamma}\real{\Hom(G, P^1)},
\end{align*}
the last homeomorphism by \prettyref{lem:regular}.
\end{proof}

\section{Proof of \prettyref{thm:univintro} and $S_n$-universality}
\label{sec:univ}

In this section we turn to the proof of \prettyref{thm:univintro}
(restated in its more precise form as \prettyref{thm:univn}), which
says that any free $S_n$ space $X$ can be approximated up to
$S_n$-homotopy type as $\Hom(K_n, G)$ for some suitably chosen
loopless graph~$G$,
where $S_n$ acts on this complex via $S_n\isom\Aut(K_n)$.  The graph
$G$ will be constructed by taking the face poset of $\tilde X$, a
sufficiently large subdivision of $X$, and then applying the
construction $G \deq K_n \times_{S_n} \tilde X^1$.  The result then
follows from
\[\Hom(K_n, K_n \times_{S_n} \tilde X^1) \simeq_{S_n} \Hom(K_n, \tilde X^1) \simeq_{S_n} \Hom({\bf 1}, \tilde X^1) \approx X,\]
the first homotopy equivalence coming from \prettyref{thm:secondentry}
and the second from \prettyref{cor:Kn1}.  The details are below, but
we point out that the $S_n$-action on $\Hom(K_n, \tilde X^1)$
in this chain of equivalences will use
the actions of $S_n$ on both $K_n$ and~$X$.   Seen as an $S_n$-space via
the $S_n\isom\Aut(K_n)$ action, the complex $\Hom(K_n, \tilde X^1)$ is $S_n$-homotopy
equivalent to $X$ with the trivial $S_n$ action, as we will see in \prettyref{cor:Kn1}.

Universality results of this kind are also found in \cite{Cs05} and \cite{Docuni}.  In \cite{Cs05}, Csorba establishes the case $n=2$ of our result: given a simplicial complex $X$ with a free ${\mathbb Z}_2$-action, there exists a graph~$G_X$ and a (simple) ${\mathbb Z}_2$-homotopy equivalence $\Hom(K_2,G_X) \simeq_{{\mathbb Z}_2} X$ (see also \cite{Ziv05} for an independent proof of this result).  Csorba's construction of $G_X$ involves taking the 1-skeleton of the barycentric subdivision of $X$ and adding additional edges between each vertex $v$ and the neighbors of $\gamma v$ (for $\gamma \in {\mathbb Z}_2$ the nonidentity element).  One can check that the resulting graph is isomorphic to the graph $K_2 \times_{{\mathbb Z}_2} (\Chain FX)^1$, and in \prettyref{thm:univ2} we show that this subdivision suffices to establish the $n = 2$ case of the Theorem.

We begin by establishing \prettyref{thm:GP}, which states that for certain
graphs $T$ and~$G$ the homotopy type of $\real{\Hom(T, G)}$ does not
change if loops are added to all vertices of~$T$.  Since
multi-homomorphisms send looped vertices to cliques of looped
vertices, we will have to produce such cliques near given sets of
vertices.  We start with some notation.

\begin{nota}
For a graph~$G$ and $M\subset V(G)$ we set
\[\nu(M)\deq\set{v\in V(G)\colon\text{$(u,v)\in E(G)$ for all $u\in M$}}.\]
\end{nota}

\begin{rem}
We have $\nu M\supset N$ if and only if $(u,v)\in E(G)$ for all $u\in
M$ and $v\in N$, so this relation is symmetric in $M$ and $N$.
In particular $\nu^2M=\nu(\nu M)\supset M$, since $\nu M\supset\nu M$.
\end{rem}

We collect some properties of~$\nu$ that we will need.  Note that in particular
it follows from \prettyref{it:nu1} below that $\nu M\intersect\nu^2 M$ is a
clique of looped vertices.

\begin{lem}\label{lem:nu}
Let $G$ be a graph and $M,N\subset V(G)$.
\begin{enumerate}
\item\label{it:nua}
If $\nu M\supset M$ then
$M\subset \nu M\intersect\nu^2 M$.
\item\label{it:nu1}
If $N\subset M\subset V(G)$ then $\nu(\nu
M\intersect\nu^2 M)\supset \nu N\intersect\nu^2 N$.
\item\label{it:nu2}
If $M\subset\nu N$ then
$\nu(\nu M\intersect\nu^2 M)\supset\nu N\intersect\nu^2 N$.
\item\label{it:nu3}
If $\nu M\supset N$ then
$\nu(\nu M\intersect\nu^2 M)\supset N$.
\end{enumerate}
\end{lem}

\begin{proof}\
\begin{enumerate}
\item This follows since $M\subset\nu^2 M$.
\item
$\nu(\nu M\intersect\nu^2 M)\supset\nu^2M\supset \nu^2N\supset
\nu N\intersect\nu^2 N$.
\item
$\nu(\nu M\intersect\nu^2 M)\supset\nu^3M=\nu M\supset\nu^2 N
\supset\nu N\intersect\nu^2 N$.
\item
$\nu(\nu M\intersect\nu^2 M)\supset\nu^3M=\nu M\supset N$.
\qed
\end{enumerate}\let\qedsymbol\relax
\end{proof}

\begin{defn}
We call a graph~$G$ \emph{fine} if for every
$M\subset V(G)$ with $M,\nu M\ne\emptyset$
we have $\nu M\intersect\nu^2 M\ne\emptyset$.
\end{defn}

The following key idea is from \cite[Section~7]{Ziv05}.

\begin{prop}\label{prop:GPfine}
Let $P$ be the proper part of a finite lattice and $G=\Chain(P)^1$ its comparability graph.
Then $G$ is fine.
\end{prop}

\begin{proof}
Let $M$ be a non-empty subset of~$P$, regarded as a set of vertices of~$G$.
Induction on the cardinality of $M$ shows that there are $k\ge 0$ and
$a_1\le b_1\le a_2\le b_2\le\cdots\le a_k\le b_k$, allowing
$a_1=\hat0$ and $b_k=\hat 1$, such that
$\nu M=([a_1,b_1]\unite\cdots\unite[a_k,b_k])\wo\set{\hat0,\hat1}$.
It follows that
$\set{a_1,b_1,\dots,a_k,b_k}\wo\set{\hat0,\hat1}\subset\nu M\intersect\nu^2 M$.
\end{proof}

\begin{thm}\label{thm:GP}
Let $G$ be a fine graph and $T$ a graph without isolated
vertices.  Denote by $\looped T$ the reflexive graph obtained from $T$
by adding loops to all vertices.
Then the inclusion map
\begin{equation*}
i\colon\Hom(\looped T, G)\to\Hom(T, G)
\end{equation*}
induced by the inclusion map $T\to\looped T$ induces a homotopy equivalence
of spaces with the homotopy inverse and the homotopies natural in $T$ and
with respect to automorphisms of~$G$.
\end{thm}

\begin{proof}
We define
\begin{align*}
j\colon\Chain \big(\Hom(T, G)\big)&\to\Hom(\looped T, G)\\
  (M_u^r)_{u\in V(T)}^{0\le r\le k}
&\mapsto\left(\Union_{r=0}^k(\nu M^r_u\intersect\nu^2 M^r_u)\right)_{u\in V(T).}
\end{align*}
To clarify the notation, for example the left hand side represents the chain
$\set{f_0,\dots,f_k}$ with $f_r(u)=M^r_u$, and we assume $f_0\le\dots\le f_k$.
We have to check that the image of~$j$ is actually in $\Hom(\looped T,
G)$.  Since $G$ is fine, we have $\nu M^r_u\intersect\nu^2
M^r_u\ne\emptyset$ for all $u$ and~$r$. Furthermore, $\nu(\nu
M^r_u\intersect\nu^2 M^r_u)\supset \nu M^s_v\intersect\nu^2 M^s_v$ for
all $0\le r,s\le k$ and $(u,v)\in E(\looped T)$.  For $(u,v)\in E(T)$
this follows from \ref{lem:nu}(\ref{it:nu2}), for $u=v$ it follows from
\ref{lem:nu}(\ref{it:nu1}).

We have $\big(j\cmps\Chain(i) \big)\left((M_u^r)_{u\in V(T)}^{0\le r\le k}\right)
\ge(\nu M_u^k\intersect\nu^2 M_u^k)_u\ge(M_u^k)_u$ by \ref{lem:nu}(\ref{it:nua}),
so $\real j\cmps\real i\homot\id_{\real{\Hom(\looped T, G_p)}}$.

To examine $i\cmps j$, we define an auxiliary map
\begin{align*}
h\colon\Chain \big(\Hom(T, G)\big)&\to\Hom(T, G),\\
  (M_u^r)_{u\in V(T)}^{0\le r\le k}
&\mapsto\left(M_u^k\unite\Union_{r=0}^k(\nu M^r_u\intersect\nu^2 M^r_u)\right)_{u\in V(T)}.
\end{align*}
We have to check $M^k_u\subset\nu(\nu M^r_v\intersect\nu^2 M^r_v)$ for
$(u,v)\in E(T)$.  This follows from \ref{lem:nu}(\ref{it:nu3}).  Since
$i\cmps j\le h$ and $h\left((M_u^r)_{u\in V(T)}^{0\le r\le
  k}\right)\ge(M_u^k)_u$, we have $\real i\cmps\real
j\homot\id_{\real{\Hom(T, G_P)}}$.

All maps have been natural in $T$ and with respect to automorphisms of~$G$.
\end{proof}

For the following, recall that a (finite) reflexive graph $H$ is called \emph{dismantlable} if there exists an ordering of the vertices $v_1, \dots, v_n$ such that for all $i<n$ we have $N_{H_i}(v_i) \subset N_{H_i}(v_{i+1})$, where $H_i$ is the subgraph of $H$ induced on the vertices $\{v_i, v_{i+1}, \dots, v_n\}$.  In \cite{DocHom} it is shown that if $H$ is dismantlable then for any graph $G$ the homomorphism $H \rightarrow {\bf 1}$ induces a homotopy equivalence $\Hom({\bf 1}, G) \rightarrow \Hom(H, G)$, with homotopy inverse induced by the graph homomorphism which sends ${\bf 1}$ to $v_n$.

\begin{cor}\label{cor:ldismantlable}
Let $G$ be a fine graph and $T$ a graph without isolated
vertices such that $\looped T$ is dismantlable.  Then the map
\begin{equation*}
\Hom(\ug, G)\to\Hom(T, G)
\end{equation*}
induced by the graph homomorphism $T\to\ug$ induces a homotopy
equivalence of spaces with the homotopy inverse and the homotopies
natural with respect to automorphisms of~$G$.
\end{cor}

\begin{proof}
The map $T\to\ug$ factorizes as $T\to\looped T\to\ug$.  Hence this follows from \prettyref{thm:GP} and the result from \cite{DocHom} mentioned above.
\end{proof}

\begin{cor}\label{cor:Kn1}
Let $G$ be a fine graph and $n\ge2$. Then the map
\begin{equation*}
\Hom(\ug, G)\to\Hom(K_n, G)
\end{equation*}
induced by the graph homomorphism $K_n\to\ug$ induces a homotopy
equivalence of $S_n$-spaces, where the action on $\Hom(\ug, G)$ is
trivial, with the homotopy inverse and the homotopies natural
with respect to automorphisms of~$G$.
\end{cor}

\begin{proof}
The map $K_n\to\ug$ factorizes as $K_n\to\looped K_n\to\ug$.  The
constant homomorphism $c\colon K_n\to\ug$ and the multi-homomorphism
$t\in\Hom(\ug, K_n)$, $t(\ast)=V(K_n)$, satisfy
$c\cmps t=\id\in\Hom(\ug, \ug)$ and $t\cmps c\ge\id\in\Hom(K_n, K_n)$.
\end{proof}

\begin{example}\label{ex:K2Cm}
For $m>2$ the graph~$\Cpm$ is fine.  This is easily checked
directly but also follows from \prettyref{prop:GPfine} applied to the
face poset of an $m$-gon.  Hence \prettyref{cor:Kn1} can be applied.
\prettyref{fig:Homretract} shows $\Hom(K_2, \Cpm)$ and in bold
the subposet which is the image of $\Hom({\bf 1}, \Cpm)$.  The two rows in the labels on the vertices correspond to the images of the two vertices of $K_2$.

\begin{figure}\label{fig:Homretract}.
\begin{center}
\input{Homretract.pspdftex}

\caption{Retracting $|\Hom(K_2, \Cpm)|$ onto $|\Hom({\bf 1},
\Cpm)|$.}

\end{center}
\end{figure}

\end{example}

With these preliminary results in place we now turn to the
universality theorems.  The first one is a reformulation of a result by
Csorba~\cite{Cs05} in our language.

\begin{thm}\label{thm:univ2}
Let $X$ be a finite simplicial complex with a free $\Z_2$-action.
Then $G\deq K_2\times_{\Z_2}\Chain(FX)^1$ is a loopless graph for which
$\real{\Hom(K_2, G)}\homot_{\Z_2}\real X$.
\end{thm}

\begin{rem}\label{rem:univcsorba}
We note that the vertices of $G$ correspond to vertices of the barycentric
subdivision of~$X$.  The neighborhood in~$G$ of such a vertex is composed of the vertex of the barycentric subdivision to which it is mapped under the
$\Z_2$-action, along with its neighbors there.  With this direct description of
the construction, the theorem is proved in \cite{Cs05}.
\end{rem}

We postpone the proof of \prettyref{thm:univ2}, and first establish
the following new result with $K_n$ in the place of~$K_2$.  This is
the main result of this section.

\begin{thm}\label{thm:univn}[restatement of \prettyref{thm:univintro}]
Let $X$ be a finite simplicial complex with a free $S_n$-action,
$n\ge 2$.
Then $G\deq K_n\times_{S_n}\Chain^3(FX)^1$ is a loopless graph for which
$\real{\Hom(K_n, G)}\homot_{S_n}\real X$.
\end{thm}

We will need the following lemma.  Recall that the concept of $d$-discontinuity is defined
in \prettyref{def:discont}.

\begin{lem}\label{lem:discont-subdiv}
Let $P$ be a poset with a free action of a finite group~$\Gamma$,
and let $k\ge0$.
Then the induced $\Gamma$-action on $(\Chain^{k}(P))^1$ is $2^k$-discontinuous.
\end{lem}

\begin{proof}
The proof is by induction on~$k$.
If the action is free on $P$ then it is also free on $P^1$, and is therefore is $1$-discontinuous.

Note that for a poset $P$, the graph $\Chain(P)^1$ can be regarded as
the comparability graph of~$P$.  Let $p\in P$ and $\gamma\in\Gamma$.
If $\gamma p\le p$ then since~$\Gamma$ is finite it follows that $\gamma p = p$.  Hence if the action on~$P$ is free, we see that the action on $\Chain(P)^1$ is $2$-discontinuous.

Next we show that if the action on $\Chain^k(P)^1$ is $d$-discontinuous for $k\ge1$
then the action on $\Chain^{k+1}(P)^1$ is $2d$-discontinuous.  Let $c_0,c_1,\dots,c_r$ be a
finite sequence of chains in~$\Chain^{k-1}(P)$, which represents a path in
$\Chain^{k+1}(P)^1$ from $c_0$ to~$c_r=\gamma c_0$ with $\gamma\ne e$, and
is of minimal length among such paths.  If $c_0\ge c_1$ then for any
$p\in c_1$ we have $\gamma p\in c_r$ and $\set p,
c_2,\dots,c_r,\set{\gamma p}$ is a path of the same length.  The same
construction can be applied if $c_{r-1}\le c_r$.  Together with
minimality we can therefore assume that
$c_0<c_1>c_2<\dots>c_{r-2}<c_{r-1}>c_r$, and in particular that $r=2s$ for
some integer~$s$.  Furthermore, by picking elements we can assume that
all of the $c_{2i}$ are singletons. Then $c_0,c_2,\dots,c_{2s}$ is a path in
$\Chain^k(P)^1$ from $c_0$ to~$\gamma c_{2s}$.  Hence $s\ge d$ and $r\ge2d$.
\end{proof}

\begin{proof}[Proof of \prettyref{thm:univn}]
Since $K_n$ is loopless and
since the action on $X$ (and hence that on $\Chain^3(FX)^1$) is free,
the graph $K_n\times_{S_n}\Chain^3(FX)^1$ is loopless.
From
\prettyref{prop:bundle0} we get
\[\Hom(K_n, K_n\times_{S_n}\Chain^3(FX)^1)\isom
   \Hom(K_n, K_n\times \Chain^3(FX)^1)/S_n,\]
provided that no two vertices in the same $S_n$-orbit of
$ K_n\times \Chain^3(FX)^1$ are connected by a path of length $3$ or~$4$.  Since the action on~$X$ is free,  all such paths have at least length~$8$ by
\prettyref{lem:discont-subdiv}.
Now by \prettyref{lem:bundle0'},
\begin{equation}\label{eq:univnb}
\Hom(K_n, K_n\times \Chain^3(FX)^1)/S_n\homot
\Hom(K_n,K_n)\times_{S_n}\Hom(K_n,\Chain^3(FX)^1).
\end{equation}
Both equivalences
are natural with respect to automorphisms of the~$K_n$ in coordinate, so we obtain
\begin{equation}\label{eq:univnc}
\varphi\colon \Hom(K_n, K_n\times_{S_n}\Chain^3(FX)^1)
\homot_{S_n}\Hom(K_n,K_n)\times_{S_n}\Hom(K_n,\Chain^3(FX)^1),
\end{equation}
\noindent
where the $S_n$-action on the right hand side is obtained by acting
simultaneously
on the first variables of both factors, i.e.
$\gamma[\alpha,\beta]=[\gamma\alpha, \gamma\beta]$
with $\gamma\in S_n$, $\alpha\in\Hom(K_n,K_n)$,
$\beta\in\Hom(K_n,\Chain^3(FX)^1)$.
Also note that the equivalence relation
on the right hand side is $[\alpha,\beta]=[\alpha\gamma,\beta\gamma]$ with
$(\beta\gamma)(v)=\gamma^{-1}\beta(v)$.
Since $\Hom(K_n, K_n)\isom S_n$,
where $S_n$ is regarded as a trivial poset with $n!$ disjoint elements, we have
\[\rho\colon
\Hom(K_n,K_n)\times_{S_n}\Hom(K_n,\Chain^3(FX)^1)\isom
\Hom(K_n,\Chain^3(FX)^1)\] with $\rho^{-1}(h)=[\id,h]$.  Now
$\rho(\gamma [\id,\beta])=\rho([\gamma\id,\gamma
  \beta])=\rho([\id\gamma, \gamma\beta])=\rho([\id,\gamma\beta\gamma^{-1}])$ and
therefore $\rho(\gamma x)=\gamma\cdot\rho(x)\cdot \gamma^{-1}$ for all~$x$.
So $(\gamma,\beta)\mapsto \gamma\beta\gamma^{-1}$ is the $S_n$-action on
$\Hom(K_n,\Chain^3(FX)^1)$ that we will have to keep track of.\sloppy

Finally, by \prettyref{cor:Kn1} and \prettyref{prop:GPfine}, we have
\[\real{\Hom(K_n,\Chain^3(FX)^1)}\homot_{S_n\times S_n}
\real{\Hom(\ug, \Chain^3(FX)^1)}.
\]
If $u$ is the unique vertex of~$\ug$ and
$\beta\in \Hom(\ug, \Chain^3(FX)^1)$, then
$(\gamma\beta\gamma^{-1})(u)=\gamma\cdot\beta(u\gamma)=\gamma\cdot\beta(u)$.
Since, with only slight abuse of notation
(identify a vertex of $\Chain^3(FX)^1$ with an element of
$\Chain^2(FX)$), $\beta\mapsto \beta(u)$ is an
isomorphism
\[\Hom(\ug, \Chain^3(FX)^1)\xto\isom\Chain^3(FX),\]
we obtain an $S_n$-homotopy equivalence
\[\real{\Hom(K_n, K_n\times_{S_n}\Chain^3(FX)^1)}
\homot_{S_n}
\real{\Chain^3(FX)},\]
and since $\Chain^3(FX)$ is the face poset of an iterated barycentric
subdivision of~$X$, the result follows.
\end{proof}

\begin{rem}
Following the maps in the proof, one sees that the map
that induces the homotopy equivalence is
\begin{align*}
\Chain^3(FX)&\to \Hom(K_n, K_n\times_{S_n}\Chain^3(FX)^1),\\
M&\mapsto\left(u\mapsto \set{[(u,v)]\colon v\in
  M}\right).
\end{align*}
Here we again regard vertices of $\Chain^3(FX)^1$ as elements
of~$\Chain^2(FX)$ (instead of singletons), so that an element of
$\Chain^3(FX)$ is a set of vertices of~$\Chain^2(FX)^1$.  It is also
easy to see that this map is equivariant, since for~$\gamma\in S_n$ we
have $\set{[(u,v)]\colon v\in \gamma M}=
\set{[(u,\gamma v)]\colon  v\in M}
=\set{[(u \gamma,v)]\colon v\in M}$.
\end{rem}

\begin{rem}
Since every element of $\Hom(K_n, K_n)$ is a true graph homomorphism, the
equivalence in \prettyref{eq:univnb} and hence in
\prettyref{eq:univnc} is actually an isomorphism.
\end{rem}

\begin{proof}[Proof of \prettyref{thm:univ2}]
Again, the graph $G$ is loopless, since $K_2$ is loopless and the
action on~$X$ is free (see also \prettyref{rem:univcsorba}).
Since there is no path of even length from one vertex of~$K_2$ to the other,
there is no path of even length from a vertex of $K_2\times\Chain(FX)^1$
to the other vertex of its orbit under the diagonal $\Z_2$-action.  Hence
\prettyref{prop:bundle0} can be applied to obtain
\[\Hom(K_2, K_2\times_{\Z_2}\Chain(FX)^1)
\isom
\Hom(K_2, K_2\times\Chain(FX)^1)/\Z_2.\]
The rest of the proof now proceeds exactly as for \prettyref{thm:univn}.  In short, we have
\begin{align*}
%\label{eq:p1}
\real{\Hom(K_2, K_2\times_{\Z_2}\Chain(FX)^1)}
&\homot_{\Z_2} \real{\Hom(K_2, K_2)\times_{\Z_2}\Hom(K_2, \Chain(FX)^1)}\\
%\label{eq:p2}
&\homeo_{\Z_2} \real{\Hom(K_2, \Chain(FX)^1)}\\
%\label{eq:p3}
&\homot_{\Z_2} \real{\Hom({\bf 1}, \Chain(FX)^1))}\\
&\homeo_{\Z_2}\real{X}.
\end{align*}
For details on the various actions of~$\Z_2$ we refer to
the proof of \prettyref{thm:univn}.
\end{proof}

\section{Further comments -- enriched categories}\label{sec:enrich}
One can ask is if the homotopy equivalence of posets described in \prettyref{thm:firstentry} can be seen as a true adjointness statement between the functors $\Hom(T,\qm)$ and $T\times\qm^1$.  It turns out that this is indeed the case and we sketch the construction here, focussing our attention on the non-equivariant situation ($\Gamma$ taken to be trivial).  For this it will be convenient to use the language of \emph{enriched category theory}, and we refer to \cite{Kel} for undefined terms.  The basic idea behind enriched category theory is to identify (in a natural way) the set of morphisms between any two objects in a certain category with an object in some other (monoidal) category.

As we have seen, $\Hom(G,H)$ is a way to assign a poset structure to the set of homomorphisms between graphs $G$ and $H$.  To correctly understand this in the context of an enriched category, we have to work with the following slightly modified category of posets, one which mimics the category of simplicial complexes.

\begin{defn}
Define $\mathcal P_0$ to be the category whose objects are finite
posets~$P$ with the property that for every element $x \in P$, the set
$\{y \in P\colon y \leq x, y \textrm{ is an atom}\}$ has a least upper
bound.  The morphisms of ${\mathcal P_0}$ are equivalence classes of
order-preserving maps $f\colon P
\rightarrow Q$ which take atoms to atoms, under the equivalence
relation $f \sim g$ if $f$ and $g$ agree on atoms.
\end{defn}

Note that the existence of least upper bounds implies that each
equivalence class~$[f]$ has a minimal element when regarded as a
subposet of $\Poset(P,Q)$.  Therefore, if $[f]=[g]$ then there is an
$h\in [f]$ such that $f\ge h\le g$ and hence
$\real f\homot\real h\homot\real g\colon\real P\to\real Q$.
We conclude that geometric realization $|\qm|$ is a well defined functor from ${\mathcal P_0}$ to $\mathcal{HTOP}$, the homotopy category of topological spaces.

Also note that for each poset $P \in \mathcal P_0$ we obtain a
simplicial complex $S(P)$ whose faces are given by all bounded sets of
atoms of $P$.  Similarly, the face poset $F(X)$ of a simplicial
complex $X$ gives an object in ${\mathcal P_0}$.  These two functors
determine an equivalence of categories.  Indeed, we have $S \circ F =
\id_{\mathcal S}$ and a natural isomorphism $(F \circ
S) \rightarrow \id_{\mathcal P}$.  One advantage of viewing the
category in terms of $\mathcal P_0$ is the following fact, whose proof
we leave to the reader.

\begin{lem}
There exists a categorial product $P \times Q$ in ${\mathcal P_0}$, given by the usual product of posets.  In particular the elements of $P \times Q$ are given by the underlying set $P \times Q$ with relation $(p,q) \leq (p^\prime, q^\prime)$ if $p \leq p^\prime$ and $q \leq q^\prime$.
\end{lem}

One can check that the category $\mathcal P_0$ is a \emph{cartesian category}, with final object given by the poset ${\textbf 1}$
consisting of a single element.  In particular, this gives $\mathcal P_0$ the structure of a \emph{symmetric monoidal category} (with product given by the categorical product), which we call $\mathcal P$.
The category $\mathcal P$ is also \emph{closed} in the sense that the product has a right
adjoint, provided by the following construction.  For posets $Q$ and $R$ in ${\mathcal P}$ we define $[Q,R]
\deq \Poset(Q,R)$ to be the poset of \emph{all} order-preserving maps
$f:Q \rightarrow R$ with relation $f \leq g$ if $f(x) \leq g(x)$ for
all $x \in Q$.  We then have the following adjunction, whose proof we again leave to the reader.

\begin{lem}\label{lem:Pclosed}
For any $P$, $Q$, and $R$ in ${\mathcal P}$ there exists a natural
bijection of sets
\[{\mathcal P_0}(P \times Q, R) \isom {\mathcal P_0}(P, [Q,R]).\]
\end{lem}
This gives the category ${\mathcal P}$ the structure of a $\mathcal
P$-category by setting $\mathcal P(P,Q)\deq[P,Q]$.

Returning now to our situation, if $P$ is an object of~$\mathcal P_0$
and $T$ and $G$ are finite graphs, then the proof of
\prettyref{thm:firstentry} provides a natural isomorphism
\begin{equation}\label{eq:adj-v0}
\Poset \bigl(P, \Hom(T,G)\bigr) \isom
\Hom\bigl(T \times P^1, G\bigr),
\end{equation}
in the category ${\mathcal P_0}$.  There are now two ways to interpret this as
an adjunction of functors.  For the stronger statement we first note that we
can obtain a $\mathcal P$-category $\mathcal G$ with objects finite graphs
and $\mathcal G(G,H)=\Hom(G,H)$.  We then see that $\Hom(T,\qm)\colon\mathcal G\to\mathcal P$ and
$T\times\qm^1\colon\mathcal P\to\mathcal G$ become $\mathcal P$-functors,
and \prettyref{eq:adj-v0} an adjunction
\begin{equation*}
\mathcal P\big(P,\Hom(T,G)\big)\isom\mathcal G \big(T\times P^1,G \big)
\end{equation*}
of $\mathcal P$-functors.  To obtain an adjunction of ordinary
functors, we apply the `underlying elements' functor $\mathcal P_0(\mathbf 1,\qm)\colon
\mathcal P_0\to\mathcal{SET}$ to this isomorphism.  As a special case
of \prettyref{lem:Pclosed} we obtain
$\mathcal P_0(\mathbf 1,\Poset(P,Q))\isom\mathcal P_0(P,Q)$.
On the other hand, if we denote by $\mathcal G_0$ the category of
finite graphs and graph homomorphisms, we have
$\mathcal P_0 \big(\mathbf 1, \Hom(H, G)\big)\isom\mathcal G_0(H,G)$, since the atoms of
$\Hom(H,G)$ are the graph homomorphisms from $H$ to~$G$.
Therefore we obtain
\begin{equation*}
{\mathcal P_0}\big(P, \Hom(T,G)\big) \isom {\mathcal G_0}\big(T \times P^1, G\big),
\end{equation*}
an adjunction between the (ordinary) functors
$\Hom(T, \qm)\colon\mathcal G_0\to\mathcal P_0$ and
$T\times\qm^1\colon\mathcal P_0\to\mathcal G_0$.

\bibliographystyle{plain}
\bibliography{litgraph}

\end{document}